\documentclass[11pt]{amsart} 
\usepackage[utf8]{inputenc}
\usepackage[english]{babel}
\usepackage{amsthm}
\usepackage{amsmath, mathtools, amssymb, mathrsfs}
\usepackage{adjustbox}

\usepackage[colorlinks=true,linktocpage=true,linkcolor=blue,citecolor=red]{hyperref}
\usepackage[capitalise]{cleveref}
\usepackage{comment}
\usepackage{tikz,tikz-cd}
\usepackage{xypic}
\usepackage{graphicx}
\usepackage{amsfonts}

\usepackage{xcolor}
\usepackage{enumitem}   
\usepackage[margin=2.5cm]{geometry}

\usepackage{aliascnt}
\newtheorem{theorem}{Theorem}[section]
\newtheorem{corollary}[theorem]{Corollary}
\newtheorem{proposition}[theorem]{Proposition}
\newtheorem{lemma}[theorem]{Lemma}
\theoremstyle{definition}
\newtheorem{definition}[theorem]{Definition}

\theoremstyle{definition}
\newtheorem*{definition*}{Notational Conventions}

\theoremstyle{remark}
\newtheorem{remark}[theorem]{Remark}

\theoremstyle{definition}
\newtheorem{example}[theorem]{Example}
\theoremstyle{definition}
\newtheorem{construction}[theorem]{Construction}

\theoremstyle{definition}

\theoremstyle{definition}
\newtheorem*{thma}{Theorem A}

\theoremstyle{definition}
\newtheorem*{thmb}{Theorem B}

\newcommand{\dom}{\mathrm{dom}}

\newcommand{\Ran}{\mathrm{Ran}}

\DeclareMathOperator{\Hom}{\mathrm{Hom}}

\DeclareMathOperator{\Ext}{\mathrm{Ext}}

\title{Definable Obstruction theory}
\author{Nicholas J. Meadows}

\begin{document}
\maketitle

\begin{abstract}
A series of recent papers by Bergfalk, Lupini and Panagiotopoulus developed the foundations of a field known as `definable algebraic topology,' in which classical cohomological invariants are enriched by viewing them as groups with a Polish cover. This allows one to apply techniques from descriptive set theory to the study of cohomology theories. In this paper, we will establish a `definable' version of a classical theorem from obstruction theory, and use this to study the potential complexity of the homotopy relation on the space of continuous maps $C(X, |K|)$, where $X$ is a locally compact Polish space, and K is a locally finite countable simplicial complex. We will also characterize the Solecki Groups of the \v{C}ech cohomology of X, which are the canonical chain of subgroups with a Polish cover that are least among those of a given complexity.
\end{abstract}

\section*{Introduction}

A Polish space is a separable, completely metrizable space. 
A major theme of research in descriptive set theory is the study of the relative complexity of equivalence relations on Polish spaces under the notion of Borel reducibility.
A particularly important class of equivalence relations are orbit equivalence relations (see \cite[Chapter 3]{gao_invariant_2009}). These arise from the continuous action of a Polish group (i.e. a topological group whose underlying space is Polish) on a Polish space. Specializing further, we consider the left translation action of a continuous homomorphism $\phi : H \rightarrow G$ of Polish groups. The image of such a homomorphism is called a \emph{Polishable subgroup} of $G$ and $G/ \phi(H)$ is called a \emph{group with a Polish cover}. In \cite{lupini_complexity_2022}, the author studied the \emph{Solecki groups} $s_{\alpha}^{\phi(H)}(G)$, where $\alpha < \omega_1$ is an ordinal and $s_{\alpha}^{\phi(H)}(G)$ is the smallest $\mathbf{\Pi}_{1 +\alpha +1}^{0}$ Polishable subgroup of G containing $\phi(H)$. In particular the \emph{Solecki rank}, the smallest $\alpha$ such that $s_{\alpha}^{\phi(H)}(G) = H$, was related to the potential Borel complexity of the orbit equivalence relation.

The papers \cite{bergfalk_definable_2020} and \cite{bergfalk_definable_2024} initiated the research programme known as `Definable Algebraic Topology', in which objects from algebraic topology are viewed as groups with a Polish cover and analyzed from the perspective of descriptive set theory.
 One of their main constructions was an analogue of the classical \v{C}ech cohomology functor on locally compact Polish spaces taking values in the category of groups with Polish cover. As an application, this functor was used to classify the homotopy types of mapping telescopes of d-tori, something that cannot be done with classical \v{C}ech cohomology. Indeed, there exist uncountable families of pairwise homotopy inequivalent mapping
 telescopes of d-tori on which the classical \v{C}ech cohomology is constant.

Recall that an \emph{Eilenberg-Mac Lane space of type $(G, n)$ } is a space $K( G, n)$ that satisfies:
$$
\pi_{m}(K( G, n)) = 
\left\{
	\begin{array}{ll}
		G  & \mbox{if } m = n \\
		0 & \mbox{if } otherwise
			\end{array}
\right.
$$
One of the main theorems of \cite{bergfalk_definable_2024} is the definable version of the Huber isomorphism. 
Given a locally compact Polish space $X$, a countable Polish group $G$ and $n \ge 2$, there is a `definable' bijection:
$$
[X, K( G, n)] \cong \check{H}^{n}(X; G).
$$
where $[X, K(G, n)]$ denotes the homotopy classes of continuous maps between $X$ and $\check{H}^{n}$ is the \v{C}ech cohomology functor (see \cref{prop1.8} for the sense in which definable is used here). 
Thus, we can study the potential complexity of the homotopy relation on $C(X, Y)$ (i.e. the continuous functions between Polish spaces $X$ and $Y$ equipped with the compact open topology) by studying the potential complexity of the \v{Cech} cohomology functor from \cite{bergfalk_definable_2024}, at least in the case that $Y$ is an Eilenberg-Mac Lane space. As we will see, in certain situation this reduces to the potential complexity of certain $\mathrm{Ext}$ groups, which was studied in \cite{casarosa_bootstrap_2025}.

The significance of this result is that Eilenberg-Maclane spaces are in some sense the `building blocks' of homotopy types. Indeed, every simply-connected CW complex $X$ is the inverse limit of its Postnikov tower of fibrations
$$
P_{n}(X) \xrightarrow{\phi_{n}} P_{n-1}(X) \cdots
$$ 
and there is a fiber sequence 
$$
K(\pi_{n}(X), n) \rightarrow P_{n}(X) \xrightarrow{\phi_{n}} P_{n-1}(X)
$$

One approach to obstruction theory (see \cite[Chapter 4.3]{hatcher_algebraic_2002}) yields cohomological obstructions to finding lifts of a map $f : X \rightarrow P_{n}(Y)$ to successive stages of the Postnikov tower of $Y$ (i.e. starting from the initial stage which is an Eilenberg-Maclane space). Thus, it allows one to study the homotopy classes of maps $[X, Y]$ in situations where $Y$ is more complicated than an Eilenberg-Maclane space. 

Unfortunately, the aforementioned obstruction cohomology classes are dependent on the k-invariants of a space, which are quite difficult to understand conceptually. Thus, in order to extend the analysis of the potential complexity of the homotopy relation from \cite{bergfalk_definable_2024}, we will adapt a more classical approach to obstruction theory from \cite[Chapter 18]{fomenko_homotopical_2016} and \cite[Chapter VI]{hu_homotopy_1959} to the definable setting. 

 Essentially, this approach to obstruction theory gives a series of cohomology classes whose vanishing allows us to find successive extensions (up to homotopy) of a map $f : X^{n} \rightarrow Y$ to the higher dimensional skeleta of $X$:
$$
\xymatrix
{
	\vdots & \\
	X^{n+2}  \ar@{.>}[ddr]\ar[u] & \\
	X^{n+1} \ar@{.>}[dr]\ar[u] & \\
	X^{n} \ar[r] \ar[u]   & Y
}
$$
for CW complexes $X$ and $Y$.

The first major result of this paper is a `definable' version of one of the central results of obstruction theory (c.f.  \cite[Theorem VI.16.6]{hu_homotopy_1959}):

\begin{thma}\label{ta}
Suppose that $X$ is a locally compact Polish space, and that $K$ is a locally finite countable simplicial complex. Let $ \check{H}^{m}(X; -)$ denote the $m$th \v{C}ech cohomology functor. Suppose that 
\begin{enumerate}
\item{$K$ is $n-1$-connected}
\item{$X$ has finite covering dimension}
\item{$\check{H}^{m+1}(X; \pi_{m}|K|) = \check{H}^{m}(X; \pi_{m}|K|) = 0$ for all $m \ge n+1$}
\end{enumerate}
Then there is a definable bijection $[X, Y] \rightarrow \check{H}^{n}(X;\pi_{n}K)$.
\end{thma}

It is a strengthening of the definable Huber isomorphism and allows one to study the potential Borel complexity of the homotopy relation for a much wider range of examples than is possible with the machinery from \cite{bergfalk_definable_2024}. One particularly important example, dealt with in \cref{sec7}, is the potential complexity of the homotopy relation on $C(Y , S^{n})$, where $S^{n}$ is a sphere and $Y$ is a locally compact Polish space.

The first essential ingredient of the proof is, perhaps unsurprisingly, the classical obstruction theory from \cite[Chapter 18]{fomenko_homotopical_2016}. The second, less obvious one is the observation that the construction of definable \v{C}ech cohomology from \cite{bergfalk_definable_2024} can viewed through the lens of the branch of topology known as shape theory (c.f. \cite{mardesic_shape_1982}). Shape theory is a modification of traditional homotopy theory which accommodates more general spaces than CW complexes, such as locally compact Polish spaces. 

The other main theorem of the paper is the following:

\begin{thmb}
Suppose that $G$ is a Polish abelian group and $X$ is a locally compact Polish space with exhaustion $X_{0} \subseteq X_{1} \subseteq X_{2} \cdots$, and $n \ge 2$. Then we have the following inductive formula for the Solecki subgroups of $\check{H}^{n}(X; G)$:
\begin{enumerate}
\item{$s_{0}(\check{H}^{n}(X; G)) =  \bigcap_{i \in \mathbb{N}} \mathrm{Ran}(\check{H}^{n}(X, X_{i}; G) \rightarrow \check{H}^{n}(X; G)) $.}
\item{For any successor ordinal $\alpha+1$, we have: $$s_{\alpha+1}(\check{H}^{n}(X; G))  =  \bigcap_{i \in \mathbb{N}} \Ran(s_{\alpha}\check{H}^{n}(X, X_{i}; G) \rightarrow \check{H}^{n}(X; G)) $$}
\item{For any limit ordinal $\lambda$, we have:
$$
s_{\lambda}(\check{H}^{n}(X; G)) = \bigcap_{\lambda' < \lambda}  s_{\lambda'}(\check{H}^{n}(X; G))
$$}
\end{enumerate}

\end{thmb}

This theorem is proven by using a theorem from \cite{casarosa_bootstrap_2025} about the Solecki groups of the image of cohomological functors $F : \mathcal{T} \rightarrow \mathrm{LH}(\mathbf{PMod})$, where $ \mathrm{LH}(\mathbf{PMod})$ is the left heart of the category of Polish modules (see Section 6 below).
 
The paper is organized as follows. In the first section, we will describe the descriptive set theoretic preliminaries needed to understand the paper. The section will begin by reviewing the Borel hierarchy and reducibility for equivalence relations. We will briefly describe the Solecki subgroups of a Polishable group \cite{lupini_complexity_2022} . This is essential for studying the potential complexity of the $\mathrm{Ext}$ functor. 

The second section of the paper will review concepts from the papers \cite{bergfalk_definable_2024} and \cite{bergfalk_definable_2020}. In particular, we will discuss the definable versions of cohomology and the Ext functor. We will also discuss the definable versions of the homotopy extension and universal coefficient theorems. 

In the third section, we will review a relative version of the classical obstruction theory for $CW$-complexes from \cite{fomenko_homotopical_2016}.
In the fourth section, we will discuss shape theory. 
In the fifth section, we will use our collected results about the shape theory and classical obstruction theory to prove Theorem A above. 

In the sixth section, we will characterize the Solecki groups of the cohomology. This is done by applying the theory of phantom subfunctors from \cite{casarosa_bootstrap_2025}. 
The seventh and final section will give examples of the potential Borel complexity of the homotopy relation on $C(X, |K|)$, where $|K|$ is the geometric realization of a countable locally finite simplicial complex and $X$ is a locally compact Polish space.   

 \section*{Acknowledgement}
 
 The author was supported by the Starting Grant 101077154 “Definable Algebraic Topology” from the European Research
Council awarded to Martino Lupini.

\section*{Topological Preliminaries}

In this brief section, we will briefly outline some topological terminology that will be used in the paper.
\\

A \emph{simplicial complex} $K$ is a family of finite sets that is closed downwards i.e. $\sigma \subseteq \tau \in K \implies \sigma \in K$. We call $K$ a subcomplex of $L$ if $K \subseteq L$. A \emph{simplex} or \emph{face} of a simplicial complex is an element of $\sigma \in K$;  $\sigma $ is an n-simplex if $|\sigma|= n+1$. A \emph{vertex} is a 0-simplex. We will write $K^{(n)}$ for the set of \emph{$n$-simplices} of $X$. We often also write $\mathrm{dom}(K) := \cup_{\sigma \in K} \sigma = K^{(0)}$. Given any two simplicial complexes $K, L$, a \emph{simplicial map} is a function $f : \mathrm{dom}(K) \rightarrow \mathrm{dom}(L)$ such that if $\sigma $ is a simplex of $K$, then  $f(\sigma)$ is a simplex of $L$. For each simplicial complex $K$, we will write $K^{n}$ for the n-skeleton of $K$, defined by $K^{n} := \bigcup_{m \le n} K^{(m)}$. 

A simplicial complex is \emph{finite} if it has finitely many vertices, and \emph{countable }if it has countably many vertices. It is \emph{locally finite} if each vertex belongs to finitely many simplices. 

Suppose that $K$ is a countable simplicial complex. We can assume that $\mathrm{dom}(K) \subseteq \mathbb{N}$. Let $\{ e_{p} \}_{p \in \mathbb{N}}$ be the canonical basis of $\mathbb{R}^{\mathbb{N}}$. Then, the geometric realization of $\sigma = \{ i_{0} \cdots, i_{p}\}$ is defined to be 
$$
|\sigma| = \{ t_{0}e_{i_{0}} + \cdots + t_{p}e_{i_{p}} :  t_{0}, \cdots t_{p} \in [0, 1]. t_{0} + \cdots + t_{p} = 1  \}
$$
We will write $|K|$ for the union of the geometric realizations of all simplices of $K$. Any space homeomorphic to the geometric realization of a simplicial complex is called a \emph{Polyhedron}. Any simplicial map $f$ induces a continuous function of geometric realizations, which we denote $|f|$. It is given by the formula:
$$
|f|(t_{0}e_{i_{0}} + \cdots + t_{p}e_{i_{p}}) = t_{0}f(e_{i_{0}}) + \cdots + t_{p}f(e_{i_{p}})
$$

For a vertex $v \in \dom(K)$, we will write $\mathrm{St}_v(K)$ for the unions of the interiors of the geometric realizations of simplices containing $v$. In particular, we have:
$$
\mathrm{St}_v(K) = \bigl\{ \sum_{w \in \mathrm{dom}(K)} t_{i}a_{w} : a_{v} > 0 \bigr\} \cap |K|
$$
\\

An \emph{open cover} of a space $X$ is a collection $\mathcal{U}$ of open sets such that $\cup_{U \in \mathcal{U}} U = X$. An open cover is \emph{locally finite} if for each $x \in X$ there is a neighborhood that intersects only finitely many elements of the cover. Given a covering $\mathcal{U}$ of $X$ and $K \subseteq X$, write $K \upharpoonright \mathcal{U} := \{ U \in \mathcal{U} : U \cap K \neq \emptyset\}$. A \emph{refinement map} between covers of $X$ is a map $F : \mathcal{V} \rightarrow \mathcal{U}$ such that $U \subseteq F(U)$. If a refinement map exists, we say that $\mathcal{V}$ refines $\mathcal{U}$. We will denote by $Cover(X)$, the category of all coverings of $X$ and coverings of $X$ and refinement maps. If $\mathcal{V}$ refines $\mathcal{U}$, we write $\mathcal{U} \prec \mathcal{V}$.

 Recall that the \emph{nerve} of a cover $\mathcal{U}$ is the simplicial complex which has an n-simplex for each nonempty intersection $U_{1} \cap U_{2} \cdots \cap U_{n}$ of elements of the cover, we write it as $N(\mathcal{U})$. It should be noted that each refinement map $F$ induces a map of nerves $N(F) : N(\mathcal{V}) \rightarrow N(\mathcal{U})$.
 \\

We will write $S^{n}, D^{n}$, respectively, for the $n$-sphere and the $n$-disk. We will often write $A \approx B$ to indicate a homeomorphism between $A, B$.
We will write $\pi_{n}(X)$ for the $n$th homotopy group of a space $X$. Given two topological spaces $X, Y$, we will write $[X, Y]$ for the homotopy classes of maps from $X$ to $Y$. Similarly, if we have two topological pairs $(X, L), (Y, M)$, write $[(X, L), (Y, M)]$ for the relative homotopy classes of maps between the two pairs. We will write $f \simeq g$ if $f, g$ are homotopic maps of spaces (or pairs). We will write $C(X, Y)$ ($C((X, L), (Y, M))$) for the continuous maps between $X, Y$, ($(X, L), (Y, M)$) equipped with the compact-open topology.

 Given a CW complex $X$, we will write $X^{n}$ for its n-skeleton. A map $f : X \rightarrow Y$ of CW complexes is \emph{cellular} if $f(X^{n}) \subseteq Y^{n}$ for each $n \in \mathbb{N}$. If $G$ is an abelian group and $(X, L)$ is a CW pair we will write $H^{n}(X; G)$ ($H^{n}(X, L; G)$) for the (relative) cellular cohomology of $X$ with coefficients in $G$ \cite[Section 2.2]{hatcher_algebraic_2002}. As is well-known, this is isomorphic to the usual (relative) singular cohomology.
  It is important to note that if $K$ is a simplicial complex, then $|K|$ is a CW complex with an $n$-cell for each $n$-simplex of $K$, and we have an isomorphism of groups $H^{n}(|K|; G) \cong H^{n}(K; G)$ between the cellular cohomology of the geometric realization and the simplicial cohomology of the complex (as in \cref{con2.3} below). We will also write  $\mathcal{C}^{n}(X; G)$
 ($\mathcal{C}^{n}(X, L; G)$), $\mathcal{Z}^{n}(X; G)$
 ($\mathcal{Z}^{n}(X, L; G)$),  $\mathcal{B}^{n}(X; G)$
 ($\mathcal{B}^{n}(X, L; G)$),  
 for the (relative) cochains, (relative) cocycles and (relative) coboundaries in cellular cohomology.

\section{Preliminaries on Descriptive Set Theory}\label{sec1}

The purpose of this section is to briefly review some definitions from descriptive set theory that will be useful for our later discussions. Of particular importance are the Solecki subgroups. 
\\

A \emph{Polish space} is a separable, complete metrizable space. A \emph{Polish group} is a topological group whose underlying space is a Polish space. 
A \emph{complexity class} $\Gamma$ is a function $\Gamma$ which assigns each Polish space $X$ a set of Borel subsets of $X$, such that if $f : X \rightarrow Y$ is a continuous function and $V \in \Gamma(X)$ then $f^{-1}(V) \in \Gamma(Y)$. Of particular importance are the complexity classes $\mathbf{\Sigma}_{\alpha}^{0}, \mathbf{\Pi}_{\alpha}^{0},  D(\mathbf{\Sigma}_{\alpha}^{0})$ for $\alpha < \omega_1$ an ordinal (see \cite[Section 11.B]{kechris_classical_1995}).

\begin{definition}\label{def1.1}
Suppose that $X, Y$ are sets. 
Suppose that $E$ is an equivalence relation on $X$ and $F$ is an equivalence relation on $Y$. Then we say that $f : X \rightarrow Y$ is a reduction of $E$ to $F$ if 
$$
xEy \Leftrightarrow f(x)Ff(y)
$$
\end{definition}

\begin{definition}\label{def1.2}
We say that a subset of a Polish space is \emph{analytic} if it is the continuous image of a Borel set.
We say that an equivalence relation on a Polish space $X$ is Borel  (analytic) if it is Borel (analytic) as a subset of $X \times X$.

Suppose that $X, Y$ are Polish spaces, and $E, F$ are equivalence relations on $X, Y$ respectively. We say that $E$ is continuously reducible to $F$, written $E \le_{C} F$, if there exists a continuous reduction from $E$ to $F$. We say that $E$ is Borel reducible to $F$, written $E \le_{B} F$ if there exists a Borel reduction from $E$ to $F$. 

Suppose that $\Gamma$ is a complexity class. We say that an equivalence relation $E$ is of \emph{potential complexity $\Gamma$} if there exists an equivalence relation $F \in \Gamma(X \times X)$ and $E$ is Borel reducible to $F$ (see \cite[Lemma 12.5.4]{gao_invariant_2009}). We say that $E$ is \emph{$\Gamma$-definable} if $E \in \Gamma(X \times X)$.  
\end{definition}

\begin{lemma}\label{lem1.3}
(see \cite[Section 5.1]{gao_invariant_2009})

\begin{enumerate}
\item{
Suppose that $F \le_{c} E$ and $E$ is $\mathbf{\Sigma}_{\alpha}$ or $\mathbf{\Pi}_{\alpha}^{0}$. Then so is $F$.}
\item{ If $F \le_{B} E$ and $E$ is Borel, then so is $F$. }
\end{enumerate}
\end{lemma}

\begin{definition}\label{def1.4}
We say that an equivalence relation $E$ is \emph{idealistic} if there exists an assignment $[x] \mapsto \mathcal{F}_{[x]}$ mapping each $E$-class $[x]$ to a $\Sigma$-filter $\mathcal{F}_{[x]}$ of subsets of $X$ and a Borel function $\zeta : X \rightarrow X$ satisfying: 
\begin{enumerate}
\item{$xE \zeta(x)$ for each $x \in X$}
\item{For each Borel subset $A$ of $X \times X$ the Borel subset $A_{\mathcal{F}} \subseteq X$ defined by
$$
x \in A_{\mathcal{F}} \iff \{ x' \in [x] : (\zeta(x), x) \in A \} \in \mathcal{F}_{[x]}
$$ 
is Borel. 
}
\end{enumerate}
\end{definition}

\begin{remark}\label{xft5555}
The preceding definition, which comes from \cite[3.1]{bergfalk_definable_2024} is slightly nonstandard; see there for more details. 
\end{remark}

\begin{example}\label{exam1.5}
 A \emph{group with a Polish cover} is a quotient $G/H$, where $G$ is a Polish group and $H$ is the Borel image of a Polish group. 
 Consider the coset relation on a group with a Polish cover $\mathcal{R}(G/H)$ given by $x \mathcal{R}(G/H) x' \Leftrightarrow x H = x' H$ Then $\mathcal{R}(G/H)$ is idealistic, as witnessed by $\zeta = id$ and the assignment $[x] \mapsto \mathcal{F}_{x}$, with $D \in \mathcal{F}_{x} \Leftrightarrow x^{-1}D$ is a comeager subset of the Polish group $H$ (see \cite[Proposition 5.4.10]{gao_invariant_2009}).
 \end{example}
 
 \begin{example}\label{exam1.6} 
 Given two topological spaces $X, Y$ we write $C(X, Y)$ for the continuous functions from $X$ to $Y$ equipped with the compact-open topology. 
 Given two maps $f, g : X \rightarrow Y$, we say $f, g$ are homotopic if there exists : $h : X \times [0, 1] \rightarrow Y$ with $h|_{X \times \{ 0\}} = f, h|_{X \times \{ 1\}} = g$. This is an equivalence relation on  $C(X, Y)$.
 
 Given a Polish space $X$ and a locally finite countable simplicial complex $K$, $C(X, |K|)$ is a Polish space. Moreover, homotopy is an idealistic equivalence relation on  $C(X, |K|)$  by \cite[Theorem 4.15]{bergfalk_definable_2024}.
 
 \end{example}

\begin{definition}\label{def1.7}
A definable set is a pair $(X, E)$ such that $X$ is a Polish space and $E$ is an idealistic and Borel equivalence relation. A \emph{definable function} is a function $X/E \rightarrow Y/E$ that lifts to a Borel function. We let $\mathbf{DSet}$ be the category whose objects are definable sets and whose morphisms are definable functions. Composition is given by composition of functions and the identity is given by the identity function $X/E \rightarrow X/E$. 
\end{definition}

\begin{proposition}[{\cite[Proposition 3.9]{bergfalk_definable_2024}}]\label{prop1.8}

Suppose that $X/E$ and $X/F$ are definable sets and $f : X/E \rightarrow X/F$ is a definable function. Then the following assertions are equivalence
\begin{enumerate}
\item{$f$ is a bijection}
\item{$f$ is an isomorphism in $\mathbf{DSet}$}
\end{enumerate} 
\end{proposition}

The following is immediate from the definition of potential complexity above. 

\begin{corollary}\label{cor1.9}
Suppose that $f : X/E\rightarrow Y/F$ is a definable bijection and $\Gamma$ is a complexity class. Then $E$ has potential complexity $\Gamma$ iff $F$ does. 
\end{corollary}

We now define the notion of Solecki groups for a group with a Polish cover $G/H$ which are the smallest $\mathbf{\Pi^{0}_{1+\alpha+1}}$ subgroups of $G$ containing $H$. 
\begin{definition}\label{def1.10}
The \emph{Solecki groups} of a group with a Polish cover $G/H$ can be defined recursively for each ordinal $\alpha < \omega_1$ by the following formula:
 \begin{enumerate}
\item{$s_{0}^{H}(G) = \bar{H}^{G}$, the closure of $H$ in $G$.}
\item{$s_{\alpha+1}^{H}(G) = s_{1}(s_{\alpha}^{H}(G))$ for each successor ordinal $\alpha + 1$}
\item{$s_{\lambda}^{H}(G) = \cap_{\lambda' < \lambda} s_{\lambda'}^{H}(G)$ for $\lambda$ a limit ordinal}
\end{enumerate}
\end{definition}

It was proven in \cite[Theorem 2.1]{solecki_analytic_1999} that for each group with a Polish cover there is a smallest ordinal $\alpha$ such that $s_{\alpha}^{H}(G) = H$. This $\alpha$ is called the \emph{Solecki rank} of $G/H$.  In \cite[Theorem 6.1]{lupini_complexity_2022}, the Solecki rank was related to the potential complexity of the coset relation (\cref{exam1.5}).
In \cite{casarosa_bootstrap_2025}, \cite[Theorem 6.1]{lupini_complexity_2022} was further applied to study the potential complexity of $\mathrm{Ext}^{1}(A, B)$ for two countable abelian groups $A, B$, which is described as group with Polish cover in the next section.

\section{Preliminaries on Definable Algebraic Topology}\label{sec2}

In this section, we briefly review some of the constructions of \cite{bergfalk_definable_2024} and \cite{bergfalk_definable_2020}, which describe classical algebraic topology invariants as groups with Polish cover. In particular, we will explain how \v{C}ech cohomology and the $\mathrm{Ext}^{1}$ functor can be viewed as such. We will also review the definable analogues of various classical theorems in algebraic topology, such as the homotopy extension theorem. 
\\

Suppose that $A, B$ are groups. Then an extension $G$ of $B$ by $A$ is a short exact sequence 
$$
0 \rightarrow A \rightarrow G \rightarrow B \rightarrow 0
$$
Two extensions are equivalent if there is a commutative diagram
$$
\xymatrix
{
0 \ar[r] & A  \ar@{=}[d]  \ar[r] & G \ar[r] \ar[d]_{\cong} & B  \ar[r] \ar@{=}[d] & 0 \\
0 \ar[r] & A \ar[r] & G' \ar[r] & \ar[r] B & 0
}
$$
By the five lemma the middle map is necessarily an isomorphism. $\mathrm{Ext}^{1}(A, B)$ has a group structure given by the \emph{Baer sum} (see \cite[Section 3.4]{weibel_introduction_1995}).

\begin{construction}\label{con2.1}
Let $A, B$ be countable abelian groups. One may describe an extension $E$ of $B$ by $A$ by using $A$ and $B$ alone. Indeed, if $s : B \rightarrow E$ is a section of $E \rightarrow B$ the multiplication table is entirely determined by a unique function $c_{s} : B \times B \rightarrow A$ which satisfies
$$
s(x) + s(y) = s(x + y) + c_{s}(x, y)
$$
If $t : B \rightarrow E $ is another section of the same epimorphism, then the function $h : B \rightarrow A$ with  $h(x) = s(x) - t(x)$ is a witness to the fact that $s, t$ represent the same extension. A function $h : B \rightarrow A$ is a witness to the fact that $s, t$ represent the same extension in the sense that $h(x) = s(x) - t(x)$.

\end{construction}

This leads to us being able to define $\mathrm{Ext}^{1}$ as a group with a Polish cover. 

\begin{construction}\label{con2.2}
Suppose that $A, B$ are countable abelian groups. Suppose $A$ is equipped with the discrete topology. Then $A^{B \times B}$ equipped with the product topology is a Polish group.
Let $Z(A, B)$ denote the closed subgroup of $A^{B \times B}$ consisting of \emph{cocycles of B with coefficients in $A$}, which are functions $B \times B \rightarrow A$ satisfying the following relations for all $x, y, z \in B$
\begin{enumerate}
\item{$c(x, 0) = c(0, x)$}
\item{$c(x, y) + c(x+y, z) = c(x, y +z) + c(y, z)$}
\item{$c(x, y) = c(y, x)$}
\end{enumerate}

Let $C(A, B) = A^{B}$, where $A$ is endowed with the discrete topology and $A^{B}$ the product topology. 

Consider the function $\delta : C(A, B) \rightarrow Z(A, B)$ given by
$$
\delta(h)(x, y) = h(x) - h(y) + h(x, y)
$$
The $Z(A, B)/C(A, B) \cong \mathrm{Ext}^{1}(A, B)$, so that $\mathrm{Ext}^{1}(A, B)$ has the structure of a group with Polish cover. 
\end{construction}

\begin{construction}\label{con2.3}
Let $K$ be a countable simplicial complex and fix an abelian Polish group $G$ and consider, for every $n \in \mathbb{N}$, the collection
$\mathcal{C}^{n}(K; G)$
of all maps from the countable set $K^{(n)}$ to $G$. Endowed with the group operation of pointwise addition and the product topology of countably many copies of $G$, the collection $\mathcal{C}^{n}(K; G)$ forms the abelian Polish group of all $G$-valued cochains of $K$.

For each $n \ge 0$, there is a Polish cochain map
$$
\delta : \mathcal{C}^{n}(K; G) \rightarrow \mathcal{C}^{n+1}(K; G)
$$
which is given by the formula
$$
\delta(\zeta)(v_{0} \cdots, v_{n+1}) =  \Sigma_{i=0}^{n} (-1)^{i} \zeta(v_{0}, \cdots, \hat{v_{i}}, \cdots, v_{n+1})
$$
where $\hat{v_{i}}$ means omission of the ith vertex. This gives rise to the Polish cochain complex
$$
\cdots \xrightarrow{\delta}  \mathcal{C}^{n}(K; G) \xrightarrow{\delta} \mathcal{C}^{n+1}(K; G) \xrightarrow{\delta} \cdots
$$
which gives the simplicial cohomology $H^{n}(X; G), n \ge 0$ the structure of a group with Polish cover.

Given a subcomplex $L $ of $K$, we have the relative cochains $ \mathcal{C}^{n}(K, L; G) \subseteq  \mathcal{C}^{n}(K; G)   $
which are exactly the cochains which are zero when restricted to $L$. The restriction of the coboundary gives rise to the relative Polish cochan complex
$$
\mathcal{C}^{n}(K, L; G) \xrightarrow{\delta} \mathcal{C}^{n+1}(K, L; G)  \xrightarrow{\delta} \cdots 
$$
which shows that relative simplicial cohomology $H^{n}(X, L; G)$ has the structure of a group with Polish cover. 
\end{construction}

\begin{construction}\label{con2.4}
Let $n \in \mathbb{N}$, and let $\mathbb{N}^{n}$ be the space of all functions from $n$ to $\mathbb{N}$. We will write $\mathbb{N}^{< \mathbb{N}} = \bigcup_{n \in \mathbb{N}} \mathbb{N}^{n}$.
We will write $\mathcal{N} = \mathbb{N}^{\mathbb{N}}$ viewed as a Polish space with the product topology on discrete copies of $\mathbb{N}$. For each  $\alpha \in \mathbb{N}$, we will write $\alpha|n = (\alpha(0), \alpha(1), \cdots, \alpha(n-1))$. For each $\alpha \in \mathbb{N}^{< \mathbb{N}} $, we will define $\mathcal{N}_{s} := \{ \alpha \in \mathcal{N} : \alpha|n = s  \}$. Note that $\{ \mathcal{N}_{s}: s \in  \mathbb{N}^{< \mathbb{N}}  \}$ forms a basis of $\mathcal{N}$.
 
We will write $\mathcal{N}^{*}$ for the closed subspace of $\mathcal{N}$ consisting of non-decreasing functions, regarded as a Polish space. We think of this as endowed with the lexicographic order. We will write $\mathcal{N}^{*}_{s} := \mathcal{N}_{s} \cap \mathcal{N}^{*}$. We will write $\alpha \vee \beta := (max(\alpha(0), \beta(0)), max(\alpha(1), \beta(1)), \cdots)$ and $\alpha \wedge \beta := (min(\alpha(0), \beta(0)), min(\alpha(1), \beta(1)), \cdots)$
\end{construction}

\begin{definition}\label{def2.5}
We say that a collection of compact subsets
$$
X_{1} \subseteq X_{2} \cdots
$$
 of $X$ is an \emph{exhaustion} of $X$ if $\cup X_{i} = X$. 

A \emph{covering system} on a locally compact Polish space $X$ is a triple $\mathfrak{U} = (X_{n}, \mathcal{U}^{X}_{\alpha}, r^{\beta}_{\alpha})$ such that:
\begin{enumerate}
\item{$\{ X_{n}, n \in \mathbb{N} \}$ is an exhaustion of $X$.}
\item{$\{ \mathcal{U}_{\alpha}^{X}, \alpha \in \mathcal{N}^{*}\}$ is a collection of locally finite covers of $X$, cofinal in the collection of all coverings of $X$.}
\item{refinement maps $r^{\beta}_{\alpha} :  \mathcal{U}_{\beta}^{X} \rightarrow  \mathcal{U}_{\alpha}^{X} $ for each $\alpha \le \beta$, which satisfy relations $r^{\gamma}_{\beta} \circ r^{\beta}_{\alpha} = r^{\gamma}_{\alpha}, r^{\beta}_{\beta} = id$. }
\end{enumerate}
and which satisfy the axioms of \cite[Definition 2.3]{bergfalk_definable_2024}. 
\end{definition}

\begin{definition}\label{def2.6}

Suppose that $X$ is a space, $G$ is an abelian group and $n \in \mathbb{N}$. Then the \emph{\v{C}ech cohomology} of $X$ with coefficients in $G$ is defined to be:
$$
\check{H}^{n}(X; G) = \mathrm{colim}_{\mathcal{U} \in \mathrm{Cov}(X)} H^{n}(N(\mathcal{U}); G) = \mathrm{colim}_{\mathcal{U} \in \mathrm{Cov}(X)} H^{n}(|N(\mathcal{U})|; G)
$$
where $H^{n}$ in the second and third expressions are, respectively, simplicial and cellular cohomology. 
\end{definition}

\begin{construction}\label{con2.6}

Suppose that $X$ is a locally compact Polish space with covering system $\mathfrak{U}$ and $G$ is a Polish group.
We are going to show how the \v{C}ech cohomology of $X$ has the structure of a group with a Polish cover. 
 We write
 $$\mathcal{C}^{n}_{sem}(\mathfrak{U}; G) := \cup_{\alpha \in \mathcal{N}^{*}} \mathcal{C}^{n}(N(\mathcal{U}_{\alpha}^{X}); G). $$

In a slight abuse of notation, we will often write $r^{\beta}_{\alpha}$, for $N(r^{\beta}_{\alpha})$. For each $k \in \mathbb{N}$, let us write $\mathcal{U}_{\alpha|k}^{X}$ for the covering $ \mathcal{U}_{\alpha}^{X}\upharpoonright X_{k} $.

Given two objects $\eta : N(\mathcal{U}_{\alpha})^{(n)} \rightarrow G, \mu : N(\mathcal{U}_{\beta})^{(n)} \rightarrow G,$ we write $\eta \sim \mu$ if  we have
$
\eta \circ r^{\alpha \vee \beta}_{\alpha} = \mu \circ r^{\alpha \vee \beta}_{\beta} 
$. By \cite[Lemma 2.9]{bergfalk_definable_2024}, this is an equivalence relation.
We define $$\mathcal{C}^{n}(\mathfrak{U}; G)) := \mathcal{C}^{n}_{sem}(\mathfrak{U}; G)/ \sim.$$

For two  $\eta : N(\mathcal{U}_{\alpha}^{X})^{(n)} \rightarrow G, \mu : N(\mathcal{U}_{\beta}^{X})^{(n)} \rightarrow G$, we let $\delta_{k} = 0$ if 
$$\eta \circ r^{\alpha \vee \beta}_{\alpha}|_{N(\mathcal{U}_{\alpha|k}^{X})^{(n)}} = \mu \circ r^{\alpha \vee \beta}_{\beta}|_{N(\mathcal{U}_{\beta|k}^{X})^{(n)} } $$
and put $\delta_{k} = 1$ otherwise. We define 
\begin{equation}\label{eqmetric}
\rho([\eta], [\mu]) = \sum_{i=1}^{\infty} \frac{\delta^{k}}{2^{k}} 
\end{equation}

The function $\rho$ endows $\mathcal{C}^{n}(\mathfrak{U}; G) $ with a metric, hence the structure of a Polish space (by the discussion of \cite[Proposition 2.10]{bergfalk_definable_2024}), and we call this the \emph{space of n-cochains}. The usual coboundary operator for \v{C}ech cohomology induces a continuous map 
$$
\cdots
\mathcal{C}^{n}(\mathfrak{U}; G) \xrightarrow{\delta^{n-1}} \mathcal{C}^{n+1}(\mathfrak{U}; G) \xrightarrow{\delta^{n}} \cdots
$$ 
and we define $\mathcal{Z}^{n}(\mathfrak{U}; G) := \mathrm{ker}(\delta^{n}), \mathcal{D}^{n}(\mathfrak{U}; G) := \mathrm{im}(\delta^{n-1})$, which we call the \emph{spaces of cochains and coboundaries}. 

We have an isomorphism 
$$
\check{H}^{n}(X, G) \cong \mathcal{Z}^{n}(\mathfrak{U}; G) / \mathcal{D}^{n}(\mathfrak{U}; G)
$$
so that the usual \v{C}ech cohomology can be viewed as a group with Polish cover. 
\\

Suppose that Y is a closed subspace of $X$ so that there is an induced covering system $\mathfrak{V} = (\mathcal{U}_{\alpha}^{Y}, Y_{n}, r_{\alpha}^{\beta})$ on $Y$.
Analogously to the preceding definitions, we have Polish groups of relative cochains, cocycles and coboundaries $\mathcal{Z}^{n}(\mathfrak{U}, \mathfrak{V}; G) \subseteq \mathcal{Z}^{n}(\mathfrak{U}; G), \mathcal{C}^{n}(\mathfrak{U}, \mathfrak{V}; G) \subseteq \mathcal{C}^{n}(\mathfrak{U}; G), \mathcal{D}^{n}(\mathfrak{U}, \mathfrak{V}; G) \subseteq \mathcal{D}^{n}(\mathfrak{U}; G) $ as well as an isomorphism of groups:
$$
\check{H}^{n}(X, Y; G) \cong \mathcal{Z}^{n}(\mathfrak{U}, \mathfrak{V}; G ) / \mathcal{D}^{n}(\mathfrak{U}, \mathfrak{V}; G)
$$
so that the relative cohomology object can be viewed as a group with a Polish cover. 
\end{construction}

We will now mention several important theorems about definable algebraic topology obtained in \cite{bergfalk_definable_2024}:
\\

The following is a very useful result, which is used in the proof of \cref{thm5.1} below:

\begin{theorem}[{\cite[Theorem 6.4]{bergfalk_definable_2024}}]\label{thm2.6}
 Suppose that $A$ is a closed subspace of a locally compact Polish space $X$ and that $K$ is a countable, locally finite simplicial complex. Then every map $g:(A \times [0, 1])\cup (X\times \{0 \})\rightarrow |K|$ has an extension $X \times [0, 1]$. Furthermore, this extension can be chosen in a Borel fashion from g, in the sense that there exists a choice function
 $$
 C((A \times [0,1])\cup (X\times \{0 \}), |K|) \rightarrow  C(X \times [0,1], |K|)
 $$
 witnessing this assertion which is Borel.
\end{theorem}

Here we recap the definable Huber isomorphism, which can be used to provide examples of potential complexities for the homotopy relation using arguments similar to those of \cref{sec7}.

\begin{theorem}[{\cite[Theorem 5.7]{bergfalk_definable_2024}}]\label{thm2.7} 
Let $G$ be an countable abelian group and $K(G, n)$ denote the polyhedral Eilenberg space of type $(G, n)$ from \cite{bergfalk_definable_2024}. Then there is a definable bijection $\check{H}^{n}(X; G) \cong [X, K(G, n)]$. 
\end{theorem}

\section{Obstruction Theory for CW Complexes}\label{sec3}

Suppose that $(X, L)$ is a finite-dimensional CW pair and $Y$ is a CW complex that is connected and topologically simple, which is to say the action of the fundamental groups on all homotopy groups is trivial. For each $n \in \mathbb{N}$, we will write $\bar{X}^{n} = L \cup_{L^{n}} X^{n}$. 
The purpose of this section is to review the approach to obstruction theory from \cite[Chapter 18]{fomenko_homotopical_2016}, which provides a series of cohomological obstructions to finding successive (up to homotopy) extensions as in the below diagram:
$$
\xymatrix
{ 
 & \\
\bar{X}^{n+2} \ar@{.>}[ddr] \ar@{.}[u]& \\
\bar{X}^{n+1} \ar@{.>}[dr] \ar[u] & \\
\bar{X}^{n}   \ar[u] \ar[r] & Y 
}
$$
 Using obstruction theory, we will prove the following theorem, which is an analogue of \cite[Theorem VI.16.5]{hu_homotopy_1959}:
 suppose that $n \ge 2$ is an integer such that:
 \begin{enumerate}
 \item{$Y$ is $n-1$-connected}
 \item{$H^{m}(X, L; \pi_{m}(Y)) = H^{m+1}(X, L; \pi_{m}(Y)) = 0$ for $m \ge n+1$} 
 \end{enumerate}
 Then there is a bijection
 $$
 H^{n}(X, L; \pi_{n}Y) \simeq [(X, L), (Y, *)].
 $$
 We will use the same general strategy as \cite{hu_homotopy_1959}, but using the more modern language found in \cite{fomenko_homotopical_2016}. 
\\

\begin{construction}\label{con3.1}
Suppose that we choose a map $f: \bar{X}^n \rightarrow Y$. For each $n+1$ cell $\sigma \in Y$,  the boundary $\partial \sigma$ is a topological $n$-sphere, thus the element $f_{\sigma} := f|_{\partial \sigma}$ defines an element of $\pi_n(Y)$. We will write $c^{n+1}(f)$ for the $n+1$-cochain given by the formula 
\begin{equation}\label{cocycle}
c^{n+1}(f)(\sigma) = [f_{\sigma}].
\end{equation}

\end{construction}

\begin{definition}\label{def3.2}
We say that $f : \bar{X}^{n} \rightarrow Y$ is $n+1$-extensible iff there exists an extension 
\begin{equation}\label{pop}
\xymatrix
{ 
\bar{X}^{n} \ar[r] \ar[d] & Y \\
\bar{X}^{n+1}  \ar@{.>}[ur]&
}
\end{equation}

\end{definition}

\begin{lemma}\label{lem3.3} 
\begin{enumerate}
\item{$c^{n+1}(f)\in \mathcal{Z}^{n+1}(X, L; \pi_n (Y))$.}
\item{ Moreover $c^{n+1}(f) = 0 $ iff $f$ is $n+1$-extensible. }
\item{The class $c^{n+1}(f)$ is natural in $X$. Namely given a proper, cellular map of pairs $h : (Z, M)  \rightarrow (X, L)$, and a continuous map $i : Y \rightarrow W$ $$i_{\#}h^{\#}c^{n+1}(f) = c^{n+1}(i \circ f \circ h).$$}
\end{enumerate}
\end{lemma}

\begin{proof}
(1) By \cite[Theorem 18.1]{fomenko_homotopical_2016}, we have that $c^{n+1}(f) \in \mathcal{Z}^{n+1}(X; \pi_n(Y)) $, since it is also in $\mathcal{C}^{n+1}(X, L; \pi_n (Y))$ by construction, we know that $c^{n+1}(f)\in \mathcal{Z}^{n+1}(X, L; \pi_n (Y))$.

(2) Note that $c^{n+1}(f) = 0 $ iff  $c^{n+1}(f) (\sigma) = 0$ for all $\sigma \in \bar{X}^{n+1}$. But 
$0 = c^{n+1}(f) (\sigma) = [f_{\sigma}]$ iff $f_{\sigma}$ is null-homotopic, iff $f_{\sigma}$ has an extension to $\sigma$.

(3) Immediate from the definition of $c^{n+1}(f)$.
\end{proof}

\begin{construction}\label{con3.4}
Suppose that $f, g : \bar{X} \rightarrow Y$ are two maps such that there is a homotopy $ h : f|_{\bar{X}^{n-1}} \simeq g|_{\bar{X}^{n-1}}$. An arbitrary n-dimensional cell $\sigma$
of $\bar{X}^{n}-L$ determines a map $$k_{\sigma} := S^{n} \approx  (\partial \sigma \times [0, 1]) \cup_{(\partial \sigma \times \{ 0, 1\})} (\sigma \times  \{ 0, 1\} ) \xrightarrow{(h, (f, g))} Y.$$ 

We will define the \emph{deformation cochain} $$d^{n}(f, g; h) \in \mathcal{C}^{n}(X, L; \pi_{n}(Y))$$ to be the cochain whose value on $\sigma$ is $k_{\sigma} : S^{n} \rightarrow Y$. In the case that $h = id$, we will simply abbreviate it as $d^{n}(f, g)$ and call it the \emph{difference cochain}. 
\end{construction}

\begin{lemma}\label{lem3.5}
Suppose that $f, g : \bar{X}^{n} \rightarrow Y$ are two maps such that $h: f|_{\bar{X}^{n-1}} \simeq g|_{\bar{X}^{n-1}}$. Then the deformation cochain has the following properties:
\begin{enumerate}
\item{$d^{n}(f, g; h) = 0$ iff $f, g$ are homotopic relative to $L$ (via a homotopy extending $h$).}
\item{$\delta^{n}d^{n}(f, g; h) = c^{n+1}(f) - c^{n+1}(g)$.}
\item{For each $f : \bar{X}^{n} \rightarrow Y$ and $\alpha \in \mathcal{C}^{n}(X, L; \pi_{n}(Y))$, there is $g : \bar{X}^{n} \rightarrow Y$ such that $d^{n}(f, g) = \alpha$.}
\item{The class $d^{n}(f, g; h)$ is natural in $X$. Namely given a proper, cellular map of pairs $j : (Z, M)  \rightarrow (X, L)$, and a continuous map $i : Y \rightarrow W$: $$i_{\#}h^{\#}d^{n}(f, g; h) = d^{n}(i \circ f \circ j, i \circ g \circ j; i \circ h \circ j).$$}
\item{$d^{n}(f, g; h_{1}) + d^{n}(f, g, h_{2}) = d^{n}(f, g; h_{1} + h_{2})$.}
\end{enumerate}
\end{lemma}

\begin{proof}

(1), (4), (5) are immediate from the definition of $d^{n}(f, g; h)$. 

(2) We will assume for simplicity that both $f, g$ differ on a single cell $e$. Let $\sigma$ be an $n+1$-cell of $X$. By the definition of the boundary map in the cellular chain complex, we want to show that:
$$
c^{n+1}(f)(\sigma) - c^{n+1}(g)(\sigma) = [\sigma : e] d^{n}(f, g; h)(e)
$$
where $[\sigma : e]$ is the degree of the map:
$$
\Phi_{\sigma, e} : S^{n} \xrightarrow{\phi} \bar{X}^{n} \rightarrow \bar{X}^{n}/ (\bar{X}^{n}-\sigma) \simeq S^{n}
 $$
where $\phi$ is the attaching map of $e$. The cocycles $c^{n+1}(f)$,  $c^{n+1}(g)$ differ by $ d^{n}(f, g; h)$ on $e$. Thus, by the construction of $c^{n+1}(f)$, we see that 
$c^{n+1}(f)(\sigma) - c^{n+1}(g)(\sigma) = \Phi_{\sigma, e} \circ d^{n}(f, g; h) =  [\sigma : e]d^{n}(f, g; h)$.

(3) Suppose that $\sigma$ is an $n$-cell of $\bar{X}^{n}$. Note that the boundary $F_{\sigma} : = (\sigma \times 0) \cup (\partial \sigma \times [0, 1]) \cup  (\sigma \times 1) $ of $\sigma \times [0, 1]$ has the homotopy type of an $n$-sphere. Thus, we can find a map $g_{\sigma} : F_{\sigma} \rightarrow K$ which is homotopic to $\alpha$. Note that every map $(\sigma \times 0) \cup (\partial \sigma \times [0, 1])  \rightarrow K$ is homotopic since the source is weakly contractible. Thus by the homotopy extension theorem, we can assume that $F_{\sigma}|_{(\sigma \times 0) \cup (\partial \sigma \times I)} =  f|_{\partial \sigma} \times I \times_{\partial \sigma} f|_{ \sigma}$.
If we define a function $g : \bar{X}^{n} \rightarrow K$
$$
g(x) =
\left\{
	\begin{array}{ll}
		f(x)  & \mbox{if } x \in \bar{X}^{n-1} \\
		g_{\sigma}(x, 1) & \mbox{if } x \in \sigma
	\end{array}
\right.
$$
then it is clear that 
$$
d^{n}(f, g) = \alpha.
$$

\end{proof}

\begin{lemma}\label{lem3.6}
$[c^{n+1}(f)] = 0$ iff we can find an extension of $f|_{\bar{X}^{n-1}}$ to $\bar{X}^{n+1}$.
\end{lemma}

\begin{proof}
If $[c^{n+1}(f)] = 0$, then we have that $c^{n+1}(f) = \delta(d)$ and there exists a map $g : \bar{X}^{n} \rightarrow K$ such that $d^{n}(f, g) = -d$ and $g|_{\bar{X}^{n-1}} = f|_{\bar{X}^{n-1}}$by \cref{lem3.5} (3). But then, we have 
$0 = \delta(d)-c^{n+1}(f)  = c^{n+1}(g)$, so that $g$ has an extension to $\bar{X}^{n+1}$. 
Conversely, if there exists a map $g : \bar{X}^{n} \rightarrow K$ which agrees with $f$ on $\bar{X}^{n-1}$, and can be extended to $\bar{X}^{n+1}$, then $c^{n+1}(g) = 0$ and we have $\delta^{n}(d_{f, g}) = c^{n+1}(f) - c^{n+1}(g) = c^{n+1}(f)$, so that $[c^{n+1}(f) ] = 0$.  

\end{proof}

\begin{lemma}\label{lem3.7} (c.f. \cite[Section 18.2]{fomenko_homotopical_2016})
Suppose that $f, g : \bar{X} \rightarrow Y$ are two maps that agree up to homotopy on $\bar{X}^{n-1}$. Then the deformation cochain is a cocycle which is zero iff $f|_{\bar{X}^{n}}$ and $g|_{\bar{X}^{n}}$ are homotopic.
\end{lemma}

\begin{proof}

To prove that the deformation cochain is a cocycle, note that both $f, g$ are defined on all of $\bar{X}$, we have $\delta(d^{n}(f, g; h)) = c^{n+1}(f) - c^{n+1}(g) = 0$, so the deformation cochain is a cocycle.

We now prove the second statement.
Let us fix a homotopy $h : f|_{\bar{X}^{n-1}} \simeq g|_{\bar{X}^{n-1}}$.
 We will define a continuous map on $\bar{X} \times \{ 0, 1 \} \cup \bar{X}^{n-1} \times [0, 1]$
$$
(x, t) \mapsto
\left\{
	\begin{array}{ll}
		f(x)  & \mbox{if } t = 0   \\
		h(x, t) &  \mbox{if } x \in \bar{X}^{n-1} \\
		g(x) & \mbox{if }  t = 1
	\end{array}
\right.
$$
One can check that the obstruction to the extension of this map to $\bar{X} \times \{ 0, 1 \} \cup \bar{X}^{n} \times [0, 1]$
 lies in $\mathcal{C}^{n+1}(\bar{X} \times [0, 1], \bar{X} \times \{ 0, 1\}; \pi_{n}(Y)) \cong \mathcal{C}^{n}(X, L; \pi_{n}(Y))$, and is nothing but $d^{n}(f, g; h)$.  The result is now immediate from \cref{lem3.6}.
\end{proof}

\begin{example}\label{exam3.8}

Suppose that $f : X \rightarrow Y$ is a map such that $Y$ is $n-1$-connected. Then there is a homotopy $h : f|_{\bar{X}^{n-1}} \simeq 0$. By the homotopy extension theorem, we can choose a map $g$ such that $f \simeq g$ and $g|_{\bar{X}^{n-1}} = 0$.  

Given such a map $g$, note that $g|_{\partial \sigma} = 0$, so that $g$ determines a map $g' : \sigma /  (\partial \sigma) \rightarrow X$. Since $\sigma / \partial \sigma \simeq S^{n}$ this determines an element $g_{\sigma} \in \pi_{n}(Y)$.  

It is easy to see that by \cref{con3.4} that we have 
$$
d^{n}(f, 0; h)(\sigma) = g_{\sigma}.
$$

It is also clear that since $K$ is $n-1$-connected, the cohomology class 
$$
\chi^{n}(f) := [d^{n}(f, 0; h)]
$$
is well-defined and independent of the choice of $h$.

\end{example}

The following lemma follows immediately from \cref{lem3.5} and \cref{lem3.7}. 

\begin{lemma}\label{chilemma}
Suppose that $Y$ is an $n-1$-connected space, and $f, g : X \rightarrow Y$ are maps.
\begin{enumerate}
\item{$\chi(f) = \chi(g)$ iff $g^{n}, f^{n}$ are homotopic}
\item{The class $\chi^{n}(f)$ is natural in $X$. Namely given a proper, cellular map of pairs $j : (Z, M)  \rightarrow (X, L)$, and a continuous map $i : Y \rightarrow W$: $$i_{*}j^{*}\chi^{n}(f) = \chi^{n}(i \circ f \circ j)$$}
\end{enumerate}
\end{lemma}

The following is an analogue of \cite[Theorem IV.16.5]{hu_homotopy_1959}.

\begin{theorem}\label{thm3.8}
Suppose that $(X, L), (Y, *)$ are pairs that satisfy the following:
\begin{enumerate}
\item{$Y$ is $n-1$-connected}
\item{$H^{m}(X, L; \pi_{m}(Y)) = 0 = H^{m+1}(X, L; \pi_{m}(Y)) $ for each $m \ge n+1$}
\item{$X, Y$ are finite-dimensional}
\end{enumerate}

Then the map $[(X, L), (Y, *)] \rightarrow H^{n}(X, L; \pi_{n}(Y))$ given by $f \mapsto \chi^{n}(f)$ is a bijection.
\end{theorem}

\begin{proof}
First, we prove surjectivity. Let $\alpha \in H^{n}(X, L; \pi_{n}(Y))$.
By \cref{lem3.5} (3), we can choose some $f$ such that $[d^{n}(f, 0; h) ]= \alpha$. We have that $c^{n+1}(f) = c^{n+1}(f) - c^{n+1}(0) = \delta (d^{n}(f, 0; h) )$, so that $[c^{n+1}(f)] = 0$. Thus, we can find an extension $g$ of $f|_{\bar{X}^{n-1}}$ to $\bar{X}^{n+1}$ by \cref{lem3.6}. Since $X$ is finite dimensional, we can use \cref{lem3.6} to further extend this up to homotopy to a map $j : (X, L) \rightarrow (Y, *)$, and  $\chi^{n}(j) = [d^{n}(j, 0; h)] = \alpha$ by \cref{lem3.7}.

We now prove injectivity. Suppose that $\chi^{n}(f) = \chi^{n}(g)$. Then we have 
 $f|_{\bar{X}^{n}} \simeq g|_{\bar{X}^{n}}$. By repeated application of \cref{lem3.7} and the second hypothesis above, we conclude $f \simeq g$. 
\end{proof}

\section{Preliminaries on Shape Theory}\label{sec4}
In this section, we will describe a few basic results from the branch of topology known as shape theory, that will be needed in the proof of \cref{thm5.1}.  Shape theory was invented to address the failure of various theorems in algebraic topology for non-CW complexes. For instance, the topologists sine curve is not weakly contractible, but the nonzero homotopy groups are trivial, witnessing the failure of Whitehead's theorem.

The approach to shape theory we discuss here is based on so-called \emph{inverse systems approach}. In this approach, we represent a space as an inverse limit of a diagram  of CW complexes in a canonical way. This diagram further defines a \emph{pro-object} in the homotopy category of polyhedra, in a functorial way which will be explained below. By working with such pro-objects, one can define versions of homotopy groups and other topological constructions that behave better for non-CW complexes.

\begin{definition}\label{def4.1}

The \emph{pro-category} \textrm{Pro}$\left( \mathcal{C}\right) $ associated
with a category $\mathcal{C}$ has cofiltering diagrams in $\mathcal{C}$ as
objects, and morphisms defined by%
\begin{equation*}
\Hom_{\mathrm{Pro}\left( \mathcal{C}\right) }\left( X,Y\right) =%
\mathrm{co\mathrm{lim}}_{s\in I}\mathrm{lim}_{t\in J}\Hom_{\mathcal{C%
}}\left( X_{s},Y_{t}\right) \text{,}
\end{equation*}%
where $I$ and $J$ are the categories indexing $X$ and $Y$, respectively
(which can be assumed to be ordered sets); see \cite[Section 2.1]
{edwards_cech_1976}. The morphisms in \textrm{Pro}$\left( \mathcal{C}\right) 
$ are called pro-morphisms, and the isomorphisms in $\mathrm{Pro}\left( 
\mathcal{C}\right) $ are called \emph{pro-isomorphisms}.

\end{definition}

 One regards $\mathcal{C}$
as a full subcategory of $\mathrm{Pro}\left( \mathcal{C}\right) $ by
identifying the objects of $\mathcal{C}$ as the pro-objects indexed by the
trivial ordered set with just one element. The ind-category $\mathrm{ind}%
\left( \mathcal{C}\right) $ of $\mathcal{C}$ is defined dually by as $%
\mathrm{pro}\left( \mathcal{C}^{\mathrm{op}}\right) ^{\mathrm{op}}$.

\begin{definition}\label{def4.2}

A \emph{level pro-morphism} between pro-objects $X$, $Y$ with the same
indexing set $I$ is a pro-morphism that is represented by a natural
transformation, namely a family of morphisms $f_{s}:X_{s}\rightarrow Y_{s}$
for $s\in I$ such that for every arrow $s\rightarrow t$ in $I$ the diagram%
\begin{equation*}
\begin{array}{ccc}
X_{s} & \overset{f_{s}}{\rightarrow } & Y_{s} \\ 
\downarrow &  & \downarrow \\ 
X_{t} & \overset{f_{t}}{\rightarrow } & Y_{t}%
\end{array}%
\end{equation*}%
commutes. By \cite[Proposition 2.1.4]{edwards_cech_1976} every pro-morphism
can be represented up to isomorphism (in the category of arrows) by a level
pro-morphism.

\end{definition}

Suppose now that $\mathcal{C}$ is a category and $\mathcal{D}$ is a full
subcategory. A $\mathcal{D}$-\emph{expansion }of $X\in \mathcal{C}$ is a
pro-object $\mathbf{X}$ over $\mathcal{D}$ together with a pro-morphism $%
i:X\rightarrow \mathbf{X}$ satisfying the following universal property:
given any pro-object $\boldsymbol{Y}$ over $\mathcal{D}$ and pro-morphism $%
f:X\rightarrow \boldsymbol{Y}$ there exists a unique morphism $g:\mathbf{X}%
\rightarrow \mathbf{Y}$ such that $g\circ i=f$. Clearly, the $\mathcal{D}$%
-expansion of $X$, when it exists, is unique up to isomorphism.
We are particularly interested in the case that $\mathcal{C} = \mathrm{Ho}(\mathbf{Poly})$ and $\mathcal{D} = \mathrm{Ho}(LC)$. 
Here, $\mathrm{Ho}(LC)$ is the category whose objects are locally compact Polish spaces and  
$$\hom_{\mathrm{Ho}(LC)}(X, Y) = [X, Y]$$ and $ \mathrm{Ho}(\mathbf{Poly})$ is the full subcategory whose objects are polyhedrons. 
\\

Let $\mathcal{U}$ be a locally finite open cover of $X$. A map $%
c:X\rightarrow \left\vert N\left( \mathcal{U}\right) \right\vert $ is \emph{%
canonical }if $c^{-1}\left( \mathrm{St}_{N\left( \mathcal{U}\right) }\left(
U\right) \right) \subseteq U$ for every open set $U$ of $\mathcal{U}$.
Canonical maps exist associated with a given locally finite open cover of a
locally compact Polish space exist \cite[Appendix 1, Section 3.1, Theorem 3]%
{mardesic_shape_1982} and are unique up to homotopy \cite[Appendix 1,
Section 3.1, Theorem 7]{mardesic_shape_1982}.

\begin{theorem}\label{thm4.3}
If $X$ is a locally compact Polish space, then $i_{X}:X\rightarrow \mathbf{SH}\left(
X\right) $ is a $\mathrm{Ho}\left( \mathbf{Poly}\right) $-expansion of $X$,
where $\mathbf{SH}\left( X\right) =\left( \left\vert N\left( \mathcal{U}%
_{\alpha }^{X}\right) \right\vert \right) _{\alpha \in \mathcal{N}^{\ast }}$
and $i_{X}$ is represented by $\left( c_{\alpha }^{X}\right) _{\alpha \in 
\mathcal{N}^{\ast }}$ for canonical maps $c_{\alpha }^{X}:X\rightarrow
\left\vert N\left( \mathcal{U}_{\alpha }^{X} \right) \right\vert $. 

This defines
a functor $\mathbf{SH}:\mathrm{LC}\rightarrow \mathrm{pro}\left( \mathrm{Ho}%
\left( \mathbf{Poly}\right) \right) $, where for a continuous function $%
f:X\rightarrow Y$ one defines $\mathbf{SH}\left( f\right) $ to be the
corresponding pro-morphism obtained from $i_{Y}\circ f$ applying the
universal property of the $\mathrm{Ho}\left( \mathbf{Poly}\right) $-expansion
of $X$. When $P$ is a polyhedron, one can set $\mathbf{SH}\left( P\right) =P$
and $i_{P}$ to be the identity morphism of $P$.

\end{theorem}

\begin{definition}\label{def4.4}
The \emph{shape category }$\mathrm{Sh}\left( \mathrm{LC}\right) $ of locally
compact Polish spaces is defined to have locally compact Polish spaces as
objects, and $\Hom_{\mathrm{Sh}\left( \mathrm{LC}\right) }\left(
X,Y\right) :=\Hom_{\mathrm{pro}\left( \mathrm{Ho}\left( \mathbf{Poly}%
\right) \right) }\left( \mathbf{SH}\left( X\right) ,\mathbf{SH}\left(
Y\right) \right) $ for objects $X,Y$ of $\mathrm{Sh}\left( \mathrm{LC}%
\right) $. By definition, if $X$ is a locally compact Polish space and $P$
is a polyhedron, then%
\begin{equation*}
\Hom_{\mathrm{Sh}\left( \mathrm{LC}\right) }\left( X,P\right) \cong 
\mathrm{col\mathrm{im}}_{\alpha \in \mathcal{N}^{\ast }}[|N\left( \mathcal{U}%
_{\alpha }^{X}\right)| ,P]\text{.}
\end{equation*}%

\end{definition}

As a particular instance of \cite[Chapter I, Section 2.3, Theorem 4]
{mardesic_shape_1982} one obtains the following:

\begin{lemma}\label{lem4.5}
Given a locally compact Polish space $X$ and a polyhedron $P$:
There is a bijection
\begin{equation*}
\lbrack X,P]\cong \Hom_{\mathrm{Sh}\left( \mathrm{LC}\right) }\left(
X,P\right)
\end{equation*}%
defined by $[f]\mapsto \mathbf{SH}\left( f\right) $ for $f\in C\left(
X,P\right) $.
\end{lemma}

\begin{lemma}\label{lem4.6}
Suppose that $X$ is a locally compact Polish space and $P$ is a polyhedron.

\begin{enumerate}
\item Suppose that $p:\left\vert N\left( \mathcal{U}_{\alpha }^{X}\right)
\right\vert \rightarrow P$ is such that $p\circ c_{\alpha }^{X}$ is
homotopic to $f$. Then for every $\alpha \leq \beta $ we have that $p\circ
r^{\beta }_{\alpha } \circ c_{\beta }^{X}$ is homotopic to $f$.

\item Suppose that $p_{\alpha }:|N\left( \mathcal{U}_{\alpha }^{X}\right)
|\rightarrow P$ and $p_{\beta }:|N\left( \mathcal{U}_{\beta }^{X}\right)
|\rightarrow P$ is such that $p_{\alpha} \circ c_{\alpha }^{X}$ is homotopic to $f$
and $p_{\beta} \circ c_{\beta }^{X}$ is homotopic to $f$. Then there exists $\gamma
\in \mathcal{N}^{\ast }$ such that $\alpha \leq \gamma $ and $\beta \leq
\gamma $ and $p_{\alpha }\circ r^{\gamma }_{\alpha }$ is homotopic to $%
p_{\beta }\circ r^{\gamma }_{\beta }$.
\end{enumerate}
\end{lemma}

\begin{proof}
(1) This follows from the fact that $r^{\beta }_{\alpha }\circ c_{\beta
}^{X} $ is a canonical map, and hence homotopic to $c_{\alpha }^{X}$.

(2) By \cref{lem4.5}
there is a bijection
\begin{equation*}
\lbrack X,P]\cong \Hom_{\mathrm{Sh}\left( \mathrm{LC}\right) }\left(
X,P\right)
\end{equation*}%
where%
\begin{equation*}
\Hom_{\mathrm{Sh}\left( \mathrm{LC}\right) }\left( X,P\right) \cong 
\mathrm{col\mathrm{im}}_{\gamma \in \mathcal{N}^{\ast }}[|N\left( \mathcal{U}%
_{\gamma }^{X}\right)| ,P]\text{.}
\end{equation*}%
As $p_{\alpha} \circ c_{\alpha }^{X}$ and $p_{\beta} \circ c_{\beta }^{X}$ represent the same
element $\mathbf{SH}\left( f\right) $ of $\Hom_{\mathrm{Sh}\left( 
\mathrm{LC}\right) }\left( X,P\right) $, the conclusion follows.
\end{proof}

The next lemma gives a sufficient condition for when a morphism $f: N(\mathcal{U}_{\alpha}^{X}) \rightarrow Y$ represents $\mathbf{Sh}(f)$ after suitable refinement. 

\begin{proposition}\label{prop4.7}
Let $X$ be a locally compact Polish space, and $K$ be a locally finite
countable simplicial complex. Suppose that $\mathcal{U}$ is a locally finite
open cover of $X$. Let $p:N\left( \mathcal{U}\right) \rightarrow K$ be a
simplicial map such that, for every $U\in \mathcal{U}$ and $x\in U$, $%
f\left( x\right) \in \mathrm{St}_{K}\left( p\left( U\right) \right) $. Then
if $c:X\rightarrow \left\vert N\left( \mathcal{U}\right) \right\vert $ is a
canonical map, then $|p| \circ c$ is homotopic to $f$.
\end{proposition}

\begin{proof}
Let $\mathcal{V}$ be a locally finite open cover of $X$ that is a refinement
of $\left\{ c^{-1}\left( \mathrm{St}_{N\left( \mathcal{U}\right) }\left(
W\right) \right) :W\in \mathcal{U}\right\} $. For $V\in \mathcal{V}$, let $%
\kappa \left( V\right) \in \mathcal{U}$ be such that $c\left( V\right)
\subseteq \mathrm{St}_{N\left( \mathcal{U}\right) }\left( \kappa \left(
V\right) \right) $. For $V\in \mathcal{V}$ and $x\in V$, we have that $%
c\left( x\right) \in \mathrm{St}_{N\left( \mathcal{U}\right) }\left( \kappa
\left( V\right) \right) $, hence $|p|c\left( x\right) \in \mathrm{St}%
_{K}\left( p\kappa \left( V\right) \right) $ and $x\in \kappa \left(
V\right) $. Thus, by hypothesis, we have that $f\left( x\right) \in \mathrm{%
St}_{K}\left( p\kappa \left( V\right) \right) $.

Suppose that $U$ is an open subset of $X$. Then we write $e_{U}$ for the characteristic function of $U$.

Let $\left( \sigma _{V}\right) _{V\in \mathcal{V}}$ be a partition of unity
associated with $\mathcal{V}$. Define the continuous function $%
g:X\rightarrow \left\vert K\right\vert $, $x\mapsto \sum_{V\in \mathcal{V}%
}\sigma _{V}\left( x\right) e_{p\kappa \left( V\right) }$. Notice that $g$
is well-defined. Indeed, if $x\in X$ and $\left\{ V\in \mathcal{V}:\sigma
_{V}\left( x\right) >0\right\} =\left\{ V_{0},\ldots ,V_{\ell }\right\} $,
then we have that, for $i\in \left\{ 0,1,\ldots ,\ell \right\} $, $x\in 
\mathrm{\mathrm{Supp}}\left( \sigma _{V_{i}}\right) \subseteq V_{i}$.\
Hence, $V_{0}\cap \cdots \cap V_{\ell }\neq \varnothing $.\ Thus, we have
that 
\begin{equation*}
\varnothing \neq c\left( V_{0}\cap \cdots \cap V_{\ell }\right) \subseteq
\cap_{i=0}^{\ell }\mathrm{St}_{N\left( \mathcal{U}\right) }\left(
p\kappa \left( V_{i}\right) \right) \text{.}
\end{equation*}%
This implies that%
\begin{equation*}
p\kappa \left( V_{0}\right) \cap \cdots \cap p\kappa \left( V_{\ell }\right)
\neq \varnothing
\end{equation*}%
and hence%
\begin{equation*}
\sum_{V\in \mathcal{V}}\sigma _{V}\left( x\right) e_{p\kappa \left( V\right)
}=\sigma _{V_{0}}\left( x\right) e_{p\kappa \left( V_{0}\right) }+\cdots
+\sigma _{V_{\ell }}\left( x\right) e_{p\kappa \left( V_{\ell }\right) }\in
\left\vert K\right\vert \text{.}
\end{equation*}%
We can define a homotopy from $f$ to $g$ by setting%
\begin{equation*}
H\left( x,t\right) :=\sum_{V\in \mathcal{V}}\sigma _{V}\left( x\right)
\left( tf\left( x\right) +\left( 1-t\right) e_{p\kappa \left( V\right)
}\right)
\end{equation*}%
which is well-defined since $f\left( x\right) \in \mathrm{St}_{K}\left(
p\kappa \left( V\right) \right) $. Analogously one can define a homotopy
from $|p| \circ c$ to $g$, concluding the proof.
\end{proof}

\begin{remark}\label{rmk4.8}
Our interest in the shape category has to do with the fact that a representative of $\mathbf{Sh}(f)$ determines, by definition, a `simplicial approximation' of a map $f : X \rightarrow P$ from a locally compact Polish space to a polyhedron. That is, a factorization
up to homotopy:
$$
\xymatrix
{
|N(\mathcal{U}_{\alpha}^{X}) \ar[r]_{f_{\alpha}} &  P \\
X \ar[ur]_{f} \ar[u]_{c_{\alpha}^{X}}&
}
$$
 
The shape category itself gives a useful way of studying such factorizations, especially as they behave under refinement (see, for example, \cref{lem4.6}).
\end{remark}

To conclude the section, we will briefly study how we can use the shape category and \cref{prop4.7} to provide a relative version of the factorization of \cref{rmk4.8}.

\begin{lemma}\label{lem4.9}
Suppose that $(X, Y)$ is a pair of locally compact Polish spaces and $(P, Q)$ is a pair of polyhedra. Then for each $\alpha \in \mathcal{N}^{*}$ there is a canonical map $c_{\alpha}^{X}$, which induces a map of pairs $(X, Y) \rightarrow (P, Q)$

\end{lemma}
\begin{proof}
We will show the square commutes up to homotopy
$$
\xymatrix
{
X  \ar[r]_{c_{\alpha}^{Y}}  & \, \, |N(\mathcal{U}_{\alpha}^Y)|\\
Y \ar[u]_{i} \ar[r]_{c_{\alpha}^{X}}& \ar[u]_{i_{\alpha}} \, \,  |N(\mathcal{U}_{\alpha}^X)|
}
$$
where $c_{\alpha}^{X}, c_{\alpha}^{Y}$ are canonical maps. 
$$
 c_{\alpha}^{X} \circ i(x) \in  c_{\alpha}^{X}(U) \subseteq St_{N(\mathcal{U}_{\alpha}^{X})}(U)  = St_{N(\mathcal{U}_{\alpha}^{Y})}(i_{\alpha}(U \cap Y))
$$
Thus, we conclude from \cref{prop4.7} that the diagram commutes up to homotopy.

The homotopy extension theorem for Polish spaces (\cref{thm2.6}) implies that we can replace $c_{\alpha}^{X}$ with a canonical map so that the above diagram strictly commutes. Hence the result. 
\end{proof}

The following is immediate from \cref{rmk4.8} and \cref{lem4.9} above:

\begin{proposition}\label{prop4.10}
Suppose that $(f,  g): (X, Y) \rightarrow (P, Q)$ is a map of pairs where the source is a pair of locally compact Polish spaces and the target is a pair of polyhedra.
Then we can factorize it up to homotopy as
$$
\xymatrix
{
(X, Y) \ar[r] \ar[d] & (|K|, |L|) \\
(|N(\mathcal{U}_{\alpha}^{X})|, |N(\mathcal{U}_{\alpha}^{Y})|) \ar[ur]&
}
$$
\end{proposition}

\section{The Main Theorem of Definable Obstruction Theory}\label{sec5}
The purpose of this section is to prove a generalization of \cref{thm3.8} in the setting definable algebraic topology. The main result is the following:

\begin{theorem}\label{thm5.1}
Suppose that $(X, Y)$ is a pair of locally compact Polish spaces, and $K$ is a locally finite, countable simplicial complex, satisfying the following conditions:
\begin{enumerate}
\item{$|K|$ is $n-1$-connected}
\item{$X$ has finite covering dimension}
\item{$\check{H}^{k+1}(X, Y; \pi_{k}|K|) = 0= \check{H}^{k}(X, Y; \pi_{k}|K|)$ for all $k \ge n+1$}
\end{enumerate}
Then there is a Borel map $\Theta_{X, Y} : C((X, Y), (|K|, *)) \rightarrow C^{n}(X, Y; \pi_{n}|K|)$ which induces a bijection
$$
[(X, Y), (|K|, *)] \cong \check{H}^{n}(X, Y; \pi_{n}|K|).
$$
Moreover, this bijection is natural in pairs $(X, Y)$ satisfying the above conditions. 
\end{theorem}

The bijection will be implemented as follows: we will find a factorization
$$
\xymatrix
{
(X, Y) \ar[d]_{(c_{\alpha}^{X}, c_{\alpha}^{Y})} \ar[r]_{f} & (|K|, *)\\
 (|N(\mathcal{U}_{\alpha}^{X})|, |N(\mathcal{U}_{\alpha}^{Y})|) \ar[ur]_{f_{\alpha}} & 
}
$$
of $f$ up to homotopy, and associate to it the deformation obstruction $d^{n}(f_{\alpha}, 0; h_{f})$. Moreover, we will show that these two operations can be done in a Borel fashion. 
\\

Throughout this section, we will often assume that a locally compact Polish space $X$ is equipped with a covering system $\mathfrak{U}$ in the sense of \ref{def2.5}. Given a closed subspace $Y$ of $X$, we will also note that there is an induced covering system $\mathfrak{V}$ on $Y$. 

We begin by defining a variant of a definition from \cite[Section 5.2]{bergfalk_definable_2024}:

\begin{definition}\label{def5.2}
 Let $S$ be the set of elements which appear in some $\mathcal{U}_{\alpha}^{X}$. Then we write $$SA_{\mathfrak{U}}(X, |K|) \subseteq \mathcal{N}^{*} \times dom(K)^{S},$$ for the set consisting of pairs $(\alpha, p)$ such that $p|_{\alpha} : N(\mathcal{U}_{\alpha}^{X}) \rightarrow K$ is a simplicial complex map and $p(U) = 0$ if $U \not\in \mathcal{U}_{\alpha}$.

We also write  $$SA_{\mathfrak{U}}^{n}(X, |K|) \subseteq \mathcal{N}^{*} \times dom(K)^{S},$$ for the set consisting of pairs $(\alpha, p)$ such that $p|_{\alpha} : N(U_{\alpha}^{X})^{n} \rightarrow K^{n}$ is a simplicial complex map and $p(U) = 0$ if $U \not\in \mathcal{U}_{\alpha}$.

We call these, respectively, the spaces of \emph{simplicial approximation} and \emph{n-skeletal approximations} of $|K|$ by $X$. 
\end{definition}

 \begin{lemma}\label{lem5.3}
 Both  $SA_{\mathfrak{U}}^{n}(X, |K|)$ and $SA_{\mathfrak{U}}(X, |K|)$ are closed subspaces of $\mathcal{N}^{*} \times dom(K)^{S}$. In particular, they are Polish spaces. 
 \end{lemma}
\begin{proof}
We will prove the statement for  $SA_{\mathfrak{U}}^{n}(X, |K|)$.  The proof for $SA_{\mathfrak{U}}(X, |K|)$ is identical.

Recall that a discrete space $D$ can be endowed with the discrete metric given by
$$
d_{D}(x, y) =
\left\{
	\begin{array}{ll}
		0  & \mbox{if } x = y \\
		1 & \mbox{if } x \neq y
	\end{array}
\right.
$$
Since $S$ is a countable set by construction (see \cite[Proposition 2.7]{bergfalk_definable_2024}), we can then define a metric on $\mathcal{N}^{*} \times dom(K)^{S}$ via the formula
$$
d((\alpha, f), (\beta, g)) = \sum_{i=1}^{\infty} \frac{d_{\mathbb{N}}(\alpha_{i}, \beta_{i}) + d_{dom(K)}(f_{i}, g_{i})}{2^{i}} 
$$
Suppose that we choose a Cauchy sequence $(\alpha^{n}, f^{n})$ in $SA_{\mathfrak{U}}^{n}(X, |K|)$, which converges to $(\alpha, f)$ in $\mathcal{N}^{*} \times dom(K)^{S}$. We want to show that $f$ takes m-simplices of $N(\mathcal{U}_{\alpha}^{X})^{n}$ to m-simplices of $K$.

Choose an $m$-simplex $\sigma$ of $N(\mathcal{U}_{\alpha}^{X})^{n}$. Let $\{ s_{i} \}$ be an enumeration of $S$. Choose an integer $N$, such that
\begin{enumerate}
\item{$\{ s_{i} : i \le N \}$ contains each vertex of $\sigma$}
\item{ $d((\alpha^{N}, f^{N}), (\alpha, f)) < \frac{1}{2^{N}}$}
\end{enumerate}
 Then we have that $f^{N}(s_{i}) = f(s_{i})$, for each vertex $s_{i}$ of $\sigma$. Since $f^{N}$ is a simplicial map, we conclude that $f(\sigma)$ spans an $m$-simplex of $K$. Hence the result.  

\end{proof}

\begin{theorem}\label{thm5.4}
There is a Borel function $\Phi_{X} : C(X, |K|) \rightarrow SA(X, |K|)$, such that for each choice of canonical map 
$$
\xymatrix
{
X \ar[r]^{f} \ar[d]_{c_{\alpha}^{X}} & |K|  \\
|N(\mathcal{U}^{X}_{\alpha})| \ar[ur]_{|\Phi_{X}(f)|}& 
}
$$
is homotopy commutative.
\end{theorem}

\begin{proof}
Take $\Phi_{X}$ to be the Borel map in the definable simplicial approximation theorem from \cite[Lemma 5.5]{bergfalk_definable_2024}. Indeed, by construction $\Phi_{X}(f)$ satisfies the hypotheses of \cref{prop4.7}, hence the result. 
\end{proof}

\begin{lemma}\label{lem5.5}
Suppose that $X, K$ satisfy the hypotheses of \cref{thm5.1}. Then there is a continuous map:
 $\Psi_{n} : SA^{n}(X, |K|) \rightarrow \mathcal{C}^{n}(\mathfrak{U}, \pi_{n}|K|)$ which takes
$$
f \mapsto d^{n}(|f|, 0; h_{|f|})
$$
\end{lemma}

\begin{proof}

Since $K$ is $n-1$-connected, we can choose a homotopy $H : id_{|K|} \rightarrow Q_{K}$, where $Q_{K}|_{|K|^{n-1}} = 0$. We claim that the assignment 
$$
f \mapsto d^{n}(|f|, 0; H \circ |f|)
$$
is continuous.

For each $s \in  \mathbb{N}^{<\mathbb{N}}$, $p : N(\mathcal{U}_{\beta})^{(n)} \rightarrow \pi_{n}(|K|)$, we define $V_{s, p}$ to be the set of all 
elements $\zeta : N(\mathcal{U}_{\alpha}^{X})^{(n)} \rightarrow \pi_{n}(K)$ of $\mathcal{C}^{n}(\mathfrak{U}, \pi_{n}|K|)$ satisfying the following two conditions
\begin{enumerate}
\item{$\alpha \in \mathcal{N}^{*}_{s}$}
\item{$\zeta |_{N(\mathcal{U}_{\alpha|s}^{X})} = p|_{N(\mathcal{U}_{\beta | s}^{X})}$}
\end{enumerate}

As noted in \cite[Section 2.4]{bergfalk_definable_2024}, the set of all $V_{s, p}$ forms a basis of the topology on $\mathcal{C}^{n}(\mathfrak{U}; \pi_{n}|K|)$, so it suffices to show that the inverse image of each $V_{s, p}$ under $\Psi_{n}$ is open. But we have
$$
\Psi_{n}^{-1}(V_{s, p}) = \left (
\mathcal{N}^{*}_{s} \times \bigcap_{\sigma \in N(\mathcal{U}_{\alpha|s}^{X})^{(n)}} \bigcup_{Q_{K}(|\sigma'|) \simeq Q_{K}(|p(\sigma)|)} \{ f \in \dom(K)^{S} : f(\sigma) = \sigma' \}   \right )   \cap SA^{n}(X, |K|)
$$
The set $ \{ f \in \dom(K)^{S} : f(\sigma) = \sigma' \}$ is open in $\dom(K)^{S}$ for fixed choice of $n$-simplices $\sigma, \sigma'$ and 
by the construction of \cite[Proposition 2.7]{bergfalk_definable_2024}, $N(\mathcal{U}_{\alpha|s}^{X})^{(n)}$ is finite, so the result follows.

\end{proof}

\begin{lemma}\label{lem5.6}
Any refinement map $r^{\beta}_{\alpha} : |N(\mathcal{U}_{\beta}^{X})| \rightarrow  |N(\mathcal{U}_{\alpha}^{X})|$ is both cellular and proper.
\end{lemma}
\begin{proof}
The map is cellular since it is the geometric realization of a map of simplicial complexes.

Since every compact subset of $|N(\mathcal{U}_{\alpha}^{X})|$ is contained in some $|N(\mathcal{U}_{\alpha | k}^{X})|$, it suffices to show that $(r^{\beta}_{\alpha})^{-1}N(\mathcal{U}_{\alpha | k}^{X})$ is compact. However, by the construction of covering systems from \cite[Section 2]{bergfalk_definable_2024}, we have:  $(r^{\beta}_{\alpha})^{-1}(|N(\mathcal{U}_{\alpha | k}^{X})|) = |N(\mathcal{U}_{\beta | k}^{X})|$.

\end{proof}

\begin{lemma}\label{lem5.7}
Suppose that $X, K$ satisfy the hypotheses of \cref{thm5.1}. Suppose also that $K^{n-1} = 0$. 
Consider the composite 
$$
\Theta_{X} := C(X, |K|) \xrightarrow{\Phi_{X}} SA(\mathfrak{U}, |K|) \xrightarrow{I_{n}} SA^{n}(\mathfrak{U}, |K|) \xrightarrow{\Psi_{n}} \mathcal{Z}^{n}(\mathfrak{U}; \pi_{n}(|K|))
$$
where $I_{n}$ is subspace inclusion. 
This composite is Borel. Moveover,  given any factorization up to homotopy 
\begin{equation}\label{factorization}
\xymatrix
{
X \ar[r]_{c_{\alpha}^{X}} \ar[dr]_{f} & |N(\mathcal{U}_{\alpha}^{X})| \ar[d]^{f_{\alpha}} \\
& |K|
}
\end{equation}
we have the equality
$$
\chi^{n}(f_{\alpha}) = [\Theta_{X}(f)]
$$
\end{lemma}

\begin{proof}
This composite is Borel by \cref{lem5.5} and \cref{thm5.4}.

By definition, we have that $\Theta_{X}(f) = \chi^{n}(f_{\alpha})$ for some $f_{\alpha}$ satisfying \cref{factorization}. It suffices to show that for any $f_{\beta}$ satisfying the above, we also have $\chi^{n}(f_{\alpha}) =\chi^{n}(f_{\beta})$ in $H^{n}(X, \pi_{n}(K))$. By \cref{lem4.6} (2), we can choose some $\gamma$ such that $f_{\alpha} \circ r^{\gamma}_{\alpha} \simeq f_{\beta} \circ r^{\gamma}_{\beta}$. By \cref{chilemma} (1), this implies that $\chi^{n}(f_{\alpha} \circ r^{\gamma}_{\alpha}) = \chi^{n}( f_{\beta} \circ r^{\beta}_{\alpha})$.

We then have the equalities in $H^{n}(|N(U_{\gamma}^{X})|, \pi_{n}(K))$ (and hence in $H^{n}(X, \pi_{n}(K))$)
$$
(r^{\gamma}_{\alpha})^{*}\chi^{n}(f_{\alpha}) = \chi^{n}(f_{\alpha}  \circ r^{\gamma}_{\alpha}) =  \chi^{n}(f_{\beta}  \circ r^{\gamma}_{\beta}) = (r^{\gamma}_{\beta})^{*}\chi^{n}(f_{\beta})
$$
by \cref{chilemma} (2).

\end{proof}

\begin{lemma}\label{lem5.8}
Suppose that $i : Y \rightarrow X$ is an inclusion of a closed subspace of a Polish space. Then we have:
$$
i^{*}[\Theta_{X}(f) ]= [\Theta_{Y}(f \circ i)].
$$
\end{lemma}
\begin{proof}
By \cref{lem4.9} and its proof, we have a homotopy commutative diagram:
\begin{equation}\label{relative}
\xymatrix
{
X \ar[r]_{c_{\alpha}^{X}} & |N(\mathcal{U}_{\alpha}^{X})| \ar[rr]_{\Phi_{X}(f)} && |K|  \\
 Y \ar[u]_{i} \ar[r]_{c_{\alpha}^{Y}} & |N(\mathcal{U}_{\alpha}^{Y})| \ar[u]_{i_{\alpha}} \ar[urr]_{\Phi_{X}(f) \circ i_{\alpha}} &&
}
\end{equation}
Thus, by \cref{lem5.7}, \cref{lem5.6} and \cref{chilemma} (2) we have the equality 
$$
[\Theta_{Y}(f \circ i)] = \chi^{n}(\Phi_{X}(f) \circ i_{\alpha}) = i_{\alpha}^{*}\chi^{n}(\Phi_{X}(f)) = i_{\alpha}^{*}[(\Theta_{X}(f))].
$$
\end{proof}

\begin{theorem}\label{thm5.9}
The map $\Theta_{X}$ from \cref{lem5.7}
 (co)restricts to a Borel function
$$
\Theta_{X, Y} : C((X, Y), (|K|, *)) \rightarrow \mathcal{Z}^{n}(\mathfrak{U}, \mathfrak{V}; \pi_{n}(K))
$$
Moreover, given a closed inclusion of locally compact Polish pairs $q : (Z ,W) \rightarrow (X, Y)$ of Polish pairs, we have the equality:
$$
q^{*}[\Theta_{X, Y}(f)] = [\Theta_{X, Y}(f \circ q)].
$$

\end{theorem}

\begin{proof}
Let $ (X, Y) \xrightarrow{(f, f|_{Y})} (|K|, *)$ be an element of $C((X, Y), (|K|, *))$. By
\cref{prop4.10} we can choose $f_{\alpha}$ as in \cref{lem5.7} to be a map such that $$
(X, Y) \xrightarrow{(c_{\alpha}^X, c_{\alpha}^Y)}
(|N(\mathcal{U}_{\alpha}^{X})|,  |N(\mathcal{U}_{\alpha}^{Y})|) \xrightarrow{(f_{\alpha}, f_{\alpha}|_{|N(\mathcal{U}_{\alpha}^{Y})|})} (|K|, *)$$ is homotopic to $f$ as a map of pairs.

As in the proof of \cref{lem5.7}, after refinement, we can assume that there is a homotopy $h : f_{\alpha} \simeq \Phi_{X}(f)$. So that we have equalities $ \Theta_{X, Y}(f) := d^{n}(\Phi_{X}(f), 0; h') = d^{n}(f_{\alpha}, 0;  h'-h|_{\bar{X}^{n-1}})$ by \cref{exam3.8}.
 But $d^{n}(f_{\alpha}, 0; h'-h|_{\bar{X}^{n-1}})$ determines an element of  $\mathcal{Z}^{n}(|N(\mathcal{U}_{\alpha}^{X})|, |N(\mathcal{U}_{\alpha}^{Y})|; \pi_{n}(K))$,  and hence an element of $\mathcal{Z}^{n}(\mathfrak{U}, \mathfrak{V}; \pi_{n}(K))$. 

The second statement is now immediate from \cref{lem5.8}. 

\end{proof}

We are now ready to conclude our section with a proof of \cref{thm5.1}. We will now briefly sketch the proof. 

First, we sketch the proof of surjectivity. Given a cocycle $q \in \check{H}^{n}(X, \pi_{n}|K|)$, which can be represented by a choice of $r \in H^{n}(|N(\mathcal{U}_{\alpha}^{X})|, \pi_{n}|K|)$, we can find a map $f_{\alpha} : |N(\mathcal{U}_{\alpha}^{X})|^{n+1} \rightarrow |K|$ such that $\chi^{n}(f) = r$. The cohomological obstruction in $H^{n+2}(|N(\mathcal{U}_{\alpha}^{X})|; \pi_{n+1}|K|)$ to extending $f_{\alpha}$ to all of $ |N(\mathcal{U}_{\alpha}^{X})|^{n+2}$ does not necessarily vanish, since $H^{m+1}(|N(\mathcal{U}_{\alpha}^{X})|; \pi_{m}|K|)$ is not necessarily $0$ for all $m \ge n+1$. But as we will see in \ref{lem5.12} below, we can use the vanishing of the cohomology of $X$, to show that the obstruction vanishes for some homotopy refinement of $f_{\alpha}$ (as defined in \cref{def5.10} below). Proceeding inductively, we show that there is a map: $f_{\gamma} : |N(\mathcal{U}_{\gamma}^{X})| \rightarrow |K|$ such that $f_{\gamma}^{n+1}$ is a homotopy refinement of $f_{\alpha}$. We can then show using \cref{lem5.7} above that $[\Theta_{X, Y}(f_{\gamma} \circ c_{\gamma}^{X})] = q$. 

Similar remarks apply to the proof of injectivity. Namely, suppose that we choose two maps $f, g : (X, Y) \rightarrow (|K|, *)$ with $\Theta_{X, Y}(f) = \Theta_{X, Y}(g)$. Suppose that $\mathbf{SH}(f), \mathbf{SH}(g)$ can be represented as $f_{\alpha} \circ c_{\alpha}^{X}, g_{\alpha} \circ c_{\alpha}^{X}$, respectively. Then we show using the difference obstruction and induction on skeleta that $f_{\beta} \simeq g_{\beta}$, for suitable refinements of $f_{\alpha}, g_{\alpha}$ (see \cref{lem5.11} below). Thus, we have $\mathbf{SH}(f) = \mathbf{SH}(g)$, and so $g \simeq f$.

\begin{definition}\label{def5.10}
Suppose that $\alpha, \beta \in \mathcal{N}^*, \alpha < \beta$. Then we call a map $ (|N(\mathcal{U}_{\beta}^{X})|, |N(\mathcal{U}_{\beta}^{Y})|)  \rightarrow (|K|, *))$ a homotopy refinement of $f : (|N(\mathcal{U}_{\alpha}^{X})|, |N(\mathcal{U}_{\alpha}^{Y})|) \rightarrow (|K|, *)$ iff $g \sim f \circ r_{\alpha}^{\beta}$.
\end{definition}

\begin{lemma}\label{lem5.11}
Suppose that $X, K$ satisfy the hypotheses of \cref{thm5.1}.
Suppose that we have $(f, g), (h, i): ( |N(\mathcal{U}_{\alpha}^{X})|, |N(\mathcal{U}_{\alpha}^{Y})|) \rightarrow (|K|, *)$ are two cellular maps, such that
$f^{n} \simeq g^{n}$, then there exists homotopy refinements $f', g'$ of $f, g$ respectively, such that $f' \simeq g'$. 

\end{lemma}

\begin{proof}

We will prove by induction on $l$, that there exist cellular homotopy refinements $f_l, g_l$ of $f^{n+l}, g^{n+l}$ that are homotopic. The case $n=0$ is trivial. 

In general, suppose that such refinements exist. 
 Consider
$$
[d^{n+l}(f_l, g_l; h_{l})] \in H^{n+l}(|N(\mathcal{U}_{\beta_l}^{X})|, |N(\mathcal{U}_{\beta_l}^{Y})|; \pi_{n+l}|K|)
$$
Since $ \check{H}^{n+l}(X, Y; \pi_{n+l}|K|) = 0$, we can choose a refinement map $r^{\beta_{l+1}}_{\beta_l}$, such that 
$$
[(r^{\beta_{l+1}}_{\beta_l})^{\#} d^{n+l}(f_l, g_l; h_{l})]  = 0
$$
in $H^{n+l}(|N(\mathcal{U}_{\beta_{l+1}}^{X})|, |N(\mathcal{U}_{\beta_{l+1}}^{Y})|; \pi_{n+l}|K|)$. We have by \cref{lem5.6} and \cref{lem3.5} (3) that
$$
0 =[(r^{\beta_{l+1}}_{\beta_l})^{\#}  d^{n+l}(f_l, g_l; h_{l}) = [d^{n+l}(f_l \circ r^{\beta_{l+1}}_{\beta_l}, g_l \circ r^{\beta_{l+1}}_{\beta_l}; h_{l} \circ r^{\beta_{l+1}}_{\beta_{l}} )]
$$
so that $f_{l + 1} := f_l \circ r^{\beta_{l+1}}_{\beta_l}, g_{l+1} := g_l \circ r^{\beta_{l+1}}_{\beta_l}$ are homotopic by \cref{lem3.7}. If necessary, we can use the cellular approximation theorem to replace each map up to homotopy with a cellular one.

Since $X$ has finite covering dimension, we can assume WLOG that there exists $l$ such that $N(\mathcal{U}_{\beta})$ is $n+l$ skeletal for each $\beta > \alpha$. Thus the result follows from the statement of the first paragraph above.

\end{proof}

\begin{lemma}\label{lem5.12}
Suppose that $X, K$ satisfy the conditions of \cref{thm5.1}. Suppose that we have a map $f : |N(\mathcal{U}_{\alpha}^{X})|^{n+1} \rightarrow |K|$. Then there exists  $ \beta > \alpha$ and 
some cellular map $g : |N(\mathcal{U}_{\beta}^{X})| \rightarrow |K|$ such that $g^{n+1}$ is a homotopy refinement of $f$.
\end{lemma}
\begin{proof}
We will show by induction on $l$ that there is a map $g_{l} : |N(\mathcal{U}_{\alpha}^{X})|^{n+l} \rightarrow |K|$ such that $g_{l}^{n+1}$ is a homotopy refinement of $f$. The case $l=1$ is trivial. 

In general, suppose that we choose $g_{l} : |N(\mathcal{U}_{\beta_{l}}^{X})| \rightarrow |K|$ with the required property. we can choose some refinement map $r^{\beta_{l+1}}_{\beta_{l}}$ such that 
$$
[c^{n+l}(g_l \circ r^{\beta_{l+1}}_{\beta_{l}})] = [(r^{\beta_{l+1}}_{\beta_{l}})^{\#}(c^{n+l}(g_l) )] = 0 
$$
in $H^{n+1}(|N(\mathcal{U}_{\beta_{l+1}}^{X})|; \pi_{n+l}|K| )$ by \cref{lem3.3} (3). Thus, we can choose a homotopy extension $g_{l+1}$ of $g_{l} \circ r^{\beta_{l+1}}_{\beta_{l}}$ by \cref{lem3.6}.

Since $X$ has finite covering dimension, we can assume WLOG that there exists $l$ such that $N(\mathcal{U}_{\beta}^{X})$ is $n+l$ skeletal for each $\beta > \alpha$. Thus the result follows from the statement of the first paragraph above.

\end{proof}

We are now ready for the proof of \cref{thm5.1}:

\begin{proof}[Proof of Theorem 5.1]

We want to show that the Borel function $\Theta_{X, Y}$ from \cref{thm5.9} induces a bijection $[(X, L), (|K|, *)] \cong \check{H}^{n}(X, L ; \pi_{n}(|K|))$.

 We begin with surjectivity. Choose a cocycle $q \in \check{H}^{n}(X, Y; \pi_{n}|K|)$, represented by some $q_{\alpha} \in H^{n}(|N(\mathcal{U}_{\alpha}^{X})|, |N(\mathcal{U}_{\alpha}^{Y})|; \pi_{n}|K|)$. We can choose a map $f : |N(\mathcal{U}_{\alpha}^{X})|^{n} \rightarrow |K|$ such that $\chi^{n}(f)= q_{\alpha}$ by \cref{lem3.5} (3). We have the equality 
 $$
 0 = \delta(d^{n}(f, 0)) = c^{n+1}(f) - c^{n+1}(0) 
 $$
 Thus, we can find an extension $f' : |N(\mathcal{U}_{\alpha}^{X})|^{n+1} \rightarrow |K|$ by \cref{lem3.3} (2). Further, by applying \cref{lem5.12} we can find a cellular map $g : |N(\mathcal{U}_{\beta}^{X})| \rightarrow |K|$ such that $g^{n+1}$ is a homotopy refinement of $f$. We now have the following equalities in $\check{H}^{n}(X, Y; \pi_{n}|K|)$
 $$
 \chi^{n}(g) = \chi^{n}(g^{n}) = \chi^{n}(f \circ r^{\beta}_{\alpha}) = (r^{\beta}_{\alpha})^{*}q_{\alpha} = q.
 $$
 The first of these is by definition, the second by \cref{chilemma} (1)  and the third is by \cref{lem5.6} and \cref{chilemma} (2). By the preceding equalities and 
 \cref{lem5.7} it follows that for a canonical map $c_\alpha^{X}$,
 $\Theta_{X, Y}(g \circ c_{\alpha}^{X}) = q$. 
 
 We now will prove injectivity. Suppose that we have 
 $$
[ \Theta_{X, Y}(f)] =[\Theta_{X, Y}(g)].
 $$
By definition,  $[\Theta_{X, Y}(f)] = \chi^{n}(f_{\alpha}) $, $[\Theta_{X, Y}(g)] = \chi^{n}(g_{\alpha})$, where $f, g$ have, respectively, factorizations up to homotopy:
\begin{equation}\label{xxx}
\xymatrix
{
\ar[r]_{f} \ar[d]_{c_{\alpha}^{X}} X  & |K| & X \ar[d]_{c_{\beta}^{X}} \ar[r]_{g} & |K|  \\
|N(\mathcal{U}_{\alpha}^{X})| \ar[ur]_{f_{\alpha}} & &  |N(\mathcal{U}_{\beta}^{X})|  \ar[ur]_{g_{\alpha}} & 
}
\end{equation}
By \cref{lem4.6} (1) and \cref{chilemma} (2) we can assume that $\alpha = \beta$. By \cref{chilemma} (1), we have that $g_{\beta}^{n} \simeq f_{\beta}^{n} $. By \cref{lem5.11} it follows that $f_{\beta} \circ r^{\gamma}_{\beta}, g_{\beta} \circ r^{\gamma}_{\beta}$ are both homotopic for some choice of $\gamma$. We then have equalities
$$
f \simeq f_{\beta} \circ c_{X}^{\beta} \simeq f_{\beta} \circ r^{\gamma}_{\beta} \circ c_{\gamma}^{X} \simeq g_{\beta} \circ r^{\gamma}_{\beta} \circ c_{\gamma}^{X} \simeq g_{\beta} \circ c_{X}^{\beta} \simeq g
$$
using \cref{lem4.6} (1). 
  \end{proof}

\section{The Solecki Groups of (Definable) Cohomology}\label{sec6}

 In \cite{casarosa_bootstrap_2025}, the authors characterized the Solecki subgroups of the image of a cohomological functor $\mathcal{T} \rightarrow \mathbf{LH}(\mathbf{PMod})$, where $\mathcal{T}$ is an arbitrary triangulated category, and $\mathbf{LH}(\mathbf{PMod})$ is the left heart of Polish modules. The purpose of this section is to review the results of the aforementioned paper, and then apply them to characterize the Solecki subgroups of the definable cohomology of a locally compact Polish space.
\\

We will begin by reviewing the language of triangulated categories:

\begin{definition}\label{def6.1}
An pre-additive category is a category enriched over the category of abelian groups. An additive category is a category which is a pre-additive category with a terminal object. 

We say that an additive category is \emph{quasi-abelian} if both $\mathcal{A}$ and its dual satisfy following conditions: 
\begin{enumerate}
\item{$\mathcal{A}$ has all kernels}
\item{the class of kernels is closed under pushout}
\end{enumerate}
We say that an additive category is \emph{abelian} if it is quasi-abelian and all arrows are strict \cite[Chapter VIII]{mac_lane_categories_1998}. Equivalently, we have the following conditions:
\begin{enumerate}
\item{$\mathcal{A}$ has all kernels}
\item{every monic is a kernel}
\end{enumerate}
\end{definition}

\begin{definition}\label{def6.1}
A \emph{triangulated category} is an additive category $\mathcal{T}$ which is endowed with a additive functor $\Sigma$, together with a collection of \emph{distinguished triangles} which are diagrams in $\mathcal{T}$ of the form:
$$
A \rightarrow B \rightarrow C \rightarrow \Sigma A, 
$$
which satisfy the axioms of a triangulated category \cite[Definition 1.1.2]{neeman_triangulated_2001}, 
A functor $F : \mathcal{S} \rightarrow \mathcal{T}$ of triangulated categories is \emph{triangulated} if it maps distinguished triangles to distinguished triangles and commutes with $\Sigma$. 

Given an abelian category $\mathcal{M}$, we call a functor $F : \mathcal{T} \rightarrow \mathcal{M}$ \emph{cohomological} if it for each triangle
$$
A \rightarrow B \rightarrow C \rightarrow A[1],
$$
then
$$
F(A) \rightarrow F(B) \rightarrow F(C)
$$
is an exact sequence of abelian groups. 
\end{definition}

\begin{definition}\label{def6.2}
An \emph{exact structure} on an additive category $\mathcal{A}$ is a collection $\mathcal{E}$ of kernel-cokernel pairs $(f, g)$ that satisfy the axioms of an exact category \cite[Definition 6.1]{buhler_exact_2010}. Given $(f, g) \in \mathcal{E}$, we call $f$ an \emph{admissable monic} and $g$ an \emph{admissable epic}. The elements of $\mathcal{E}$ are called \emph{short exact sequences}. A functor of exact categories which preserves short exact sequences is called an \emph{exact functor}. 
\end{definition}

\begin{example}\label{exam6.3}

It can be shown that the class of all kernel-cokernel pairs in a quasi-abelian category is an exact structure (\cite[Proposition 4.4]{buhler_exact_2010}). Unless otherwise stated, we will assume that any exact quasi-abelian category has this structure. 
\end{example}

\begin{example}\label{exam6.4} (the derived category of an exact category)
We say that a category is \emph{idempotent complete} if it each idempotent has a kernel. Given an idempotent complete exact category $(\mathcal{A}, \mathcal{E})$, we will write $\mathrm{Ch}^{b}(\mathcal{A})$ for the category of bounded cochain complexes in $\mathcal{A}$ and $K^{b}(\mathcal{A})$ be the category whose objects are bounded complexes in $\mathcal{A}$ and whose morphisms are homotopy classes of maps of bounded complexes \cite[Section 10]{buhler_exact_2010}. 

We can equip $K^{b}(\mathcal{A})$ with the structure of a triangulated category as follows. The translation functor $\Sigma$ is given by the formula $\Sigma(A)^{n} = A^{n+1}$. We declare a triangle to be distinguished if it is isomorphic to an \emph{strict triangle} of the form:
$$
A \xrightarrow{f} B \xrightarrow{i} \mathrm{cone}(f) \xrightarrow{j} \Sigma A
$$
where $i$ and $j$ are the canonical morphisms of complexes \cite[Definition 9.2]{buhler_exact_2010}. We say that a complex $A$ is \emph{acyclic} if for each $n$, the sequence
$$
A^{n} \rightarrow A^{n+1} \rightarrow A^{n+2}
$$
is an exact sequence of admissible morphisms. 

The subcategory $N^{b}(\mathcal{A})$ of $N^{b}(\mathcal{A})$ spanned by bounded acyclic complexes is thick \cite[Corollary 10.11]{buhler_exact_2010}, namely it is strictly full and it is closed under taking direct summands of its objects \cite[Corollary 10.11]{buhler_exact_2010}. One can thus define the derived category to be the Verdier quotient : 
$$
D^{b}(\mathcal{A}) := K^{b}(\mathcal{A}) / N^{b}(\mathcal{A})
$$
(see \cite[Chapter 2]{neeman_triangulated_2001}).
By definition, this is the localization of the abelian category $K^{b}(\mathcal{A})$ at the set of quasi-isomorphisms. In particular, this means that there is a quotient functor
$Q : K^{b}(\mathcal{A}) \rightarrow D^{b}(\mathcal{A})$
that is universal among all triangulated functors $K^{b}(\mathcal{A}) \rightarrow \mathcal{T}$ that invert quasi-isomorphisms. 

Similarly, we can define $K^{+}(\mathcal{A}), K^{-1}(\mathcal{A})$ the categories of non-negatively and negatively graded bounded chain complexes, as well as the quotients $D^{+}(\mathcal{A}), D^{-}(\mathcal{A})$.
 \end{example}

\begin{example}\label{exam6.5}
Suppose that $\mathcal{A}$ is a quasi-abelian category. Then $\mathcal{A}$ is idempotent complete exact category, so that $D^{b}(\mathcal{A})$ exists. Following \cite{schneiders_quasi-abelian_1999}, we define the \emph{left heart} of $\mathcal{A}$ to be $D^{+}(\mathcal{A}) \cap D^{-}(\mathcal{A})$. Alternatively, $\mathbf{LH}(\mathcal{A})$ is the subcategory of $D^{b}(\mathcal{A})$ which is spanned by complexes $A$ with $A^{n} = 0$ for all $n \in \mathbb{Z}-\{ 0, 1\}$ and such that $\delta^{-1}: A^{-1} \rightarrow A^{0}$ is a monomorphism. 

Consider the intelligent truncations $\tau_{\le 0} : D^{b}(\mathcal{A}) \rightarrow D^{+}(\mathcal{A}), \tau_{\ge 0} : D^{b}(\mathcal{A}) \rightarrow D^{-}(\mathcal{A})$, which preserve homology and, respectively, truncate $A$ below (above) $0$. Then for each $n \in \mathbb{N}$, we have a canonical cohomological functor 
$$
H^{n} : \mathcal{D}^{b}(\mathcal{A}) \rightarrow \mathbf{LH}(\mathcal{A})
$$
given by $H^{n}(C) = \tau_{\le 0} \tau_{\ge 0} \Sigma^{-n}C$. 
\end{example}

Let $R$ be a PID. We say that a Polish $R$-module is an $R$-module that is also a Polish abelian group, and such that $x \mapsto rx$ is continuous. We write $\mathbf{PMod}(R)$ for the category of Polish modules and continuous $R$-homomorphisms between them. 

\begin{theorem}\label{thm6.6} (\cite[Theorem 6.13]{lupini_looking_2022})
The category $\mathbf{PMod}$ is an quasi-abelian category. Furthermore $\mathbf{LH}(\mathbf{PMod})$ is equivalent to the category of abelian groups with a Polish cover. 
\end{theorem}

\begin{definition}\label{def6.5}
Suppose that $\mathcal{T}$ is a triangulated category. 
 We say that an object $C$ in $\mathcal{T}$ is \emph{compact} if for each collection of objects $A_{i}$ in $\mathcal{T}$, we have:
$$
\bigoplus_{i \in \mathbb{N}} \Hom(C, A_{i}) \cong \Hom \big(C, \bigoplus_{i \in \mathbb{N}} A_{i} \big)
$$
\end{definition}

Suppose that $X$ is an object in a triangulated category. Let $\mathcal{C}$ be a class of compact objects in $X$. Let $Y$ be given object in a triangulated category $\mathcal{T}$. Then the functor $\Hom(-, Y)$ is triangulated, so since 
$$
C \rightarrow X \rightarrow X/C
$$
is an exact triangle, we have that 
$$
\Hom(X/C, Y) \rightarrow \Hom(X, Y) \rightarrow \Hom(C, Y)
$$
is exact. We thus have the following definition:

\begin{definition}\label{def6.6} (see \cite[Section 5]{christensen_ideals_1998})
Suppose that $X, Y$ are objects in a triangulated category $\mathcal{T}$. Then we define the \emph{phantom maps} from $X$ to $Y$ to be the image of the functor 
$$
\bigcap_{C \in \mathcal{C}} \mathrm{Ker}(\Hom(X, Y) \rightarrow \Hom(C, Y)) = \bigcap_{C \in \mathcal{C}} \mathrm{Ran}(\Hom(X/C, Y)  \rightarrow \Hom(X, Y))
$$
\end{definition}

Generalizing the definition of phantom maps from the literature, \cite[Section 8.2]{casarosa_bootstrap_2025} introduced the notion of \emph{$\alpha$ Phantom subfunctor} of a cohomological functor. in the case of the functor $\Hom(-, Y)$ we obtain a notion of $\alpha$-phantom map. 
 
\begin{definition}\label{def6.9}
Suppose that $\mathcal{T}$ is a triangulated category with a class of compact objects $\mathcal{C}$. 
Suppose that $F : \mathcal{T} \rightarrow \mathcal{M}$ is a cohomological functor. Then the \emph{phantom subfunctor} of order $\alpha$ is defined inductively as follows:
$$\mathrm{Ph}^{0}(F)(X) = \bigcap_{C \in \mathcal{C}} \mathrm{Ran}(F(X/C ) \rightarrow F(X)).$$
 For each successor ordinal $\alpha$, we let 
 $$
 \mathrm{Ph}^{\alpha+1}F(X) = \bigcap_{C \in \mathcal{C}} \mathrm{Ran}(\mathrm{Ph}^{\alpha}F(X/C ) \rightarrow F(X)).
 $$
 For each limit ordinal $\lambda$, let 
 $$
 \mathrm{Ph}^{\lambda}F(X) = \cap_{\lambda'< \lambda}  \mathrm{Ph}^{\lambda'}F(X).
 $$
 
\end{definition}

\begin{lemma}\label{lem6.10} (see \cite[Section 8.2]{casarosa_bootstrap_2025})
Let $F : \mathcal{T} \rightarrow \mathcal{M}$ is a cohomological functor. 
Suppose that $X$ is an object such that $X = \mathrm{hocolim} \, X_{i}$. Then we have:
\begin{enumerate}
\item{$\mathrm{Ph}(F)= \bigcap_{i \in \mathbb{N}} \mathrm{Ran}(F(X/X_{i}) \rightarrow F(X) )$}
\item{For any successor ordinal $\alpha+1$, we have: 
$$
\mathrm{Ph}^{\alpha+1}(F(X)) = \bigcap_{i \in \mathbb{N}} \mathrm{Ran}(\mathrm{Ph}^{\alpha}(F(X/X_{i}) \rightarrow F(X) )
$$ }
\end{enumerate}
\end{lemma}

\begin{theorem}\label{thm6.11} (see \cite[Theorem 8.2]{casarosa_bootstrap_2025})
Let $R$ be a countable PID. 
Suppose that there is a triangulated functor $\mathcal{T} \rightarrow \mathbf{LH}(\mathbf{PMod}(R))$. Then if $F(X_{i})$ is countable, then $\mathrm{Ph}^{\alpha}(F(X)) = s_{\alpha}F(X)$ for each $\alpha < \omega_1$.
\end{theorem}

\begin{corollary}\label{cor6.12}
Suppose that $X$ is a locally compact Polish space, and $G$ is a countable group. 
\begin{enumerate}
\item{$s_{0}(\check{H}^{n}(X; G)) = \mathrm{Ph}(\check{H}^{n}(X; G)) = \bigcap_{i \in \mathbb{N}} \Ran(\check{H}^{n}(X, X_{i}; G) \rightarrow \check{H}^{n}(X; G)) $}
\item{For any successor ordinal $\alpha+1$, we have: $$s_{\alpha+1}(\check{H}^{n}(X; G))  = \mathrm{Ph}^{\alpha+1}(\check{H}^{n}(X; G)) = \bigcap_{i \in \mathbb{N}} \Ran(\mathrm{Ph}^{\alpha}\check{H}^{n}(X, X_{i}; G) \rightarrow \check{H}^{n}(X; G)) $$}
\item{For any limit ordinal $\lambda$, we have:
$$
s_{\lambda}(\check{H}^{n}(X; G)) = \bigcap_{\lambda' < \lambda}  \mathrm{Ph}^{\lambda'}(\check{H}^{n}(X; G))
$$
}
\end{enumerate}
\end{corollary}

\begin{proof}
Suppose that $X_{0} \subseteq X_{1} \subseteq X_{2} \cdots$ is an exhaustion of $X$. Let $\mathfrak{U}$ be a choice of covering system on $X$. Write $\mathfrak{U}_{n}$ for covering system on $X_{n}$ induced by $\mathfrak{U}$. Let us write $\mathcal{C}^{*}(\mathfrak{U}; G)$ for the Polish cochain complex from \cref{con2.6}. By the definition of homotopy colimits above, it is easy to check that 
$$\tau_{\le 0} \Sigma^{n}\mathcal{C}^{*}(\mathfrak{U}; G) = \mathrm{hocolim}_{n \in \mathbb{N}} \tau_{\le 0} \Sigma^{n} \mathcal{C}^{*}(\mathfrak{U}_{n}; G).$$
in $D^{b}(\mathbf{PMod})$. 

Consider the canonical cohomological functor $H^{n} := \tau_{\ge 0}\tau_{\le 0}\Sigma^{-n} : D^{b}(\mathbf{PMod}(\mathbb{Z})) \rightarrow \mathbf{LH}(\mathbf{PMod}(\mathbb{Z}))$ from \cref{exam6.5} above. For each 
$\tau_{\le 0} \Sigma^{n} \mathcal{C}^{*}(\mathfrak{U}_{n}; G)$, we have $$\tau_{\ge 0}\tau_{\le 0} \tau_{\le 0} \Sigma^{n} \mathcal{C}^{*}(\mathfrak{U}_{n}; G) = \tau_{\ge 0}\tau_{\le 0} \Sigma^{-n} \mathcal{C}^{*}(\mathfrak{U}_{n}; G) = \check{H}^{n}(X_{n}; G)$$ is countable, since the cohomology of a compact, second-countable metric space with coefficients in a countable group is countable. Thus, we conclude using \cref{thm6.11}.
\end{proof}

We will now give a topological interpretation of the $\alpha$-phantom maps using \cref{thm5.1}. 

\begin{definition}\label{def5.13}
Suppose that $X$ is a locally compact Polish space with exhaustion $X_{0} \subseteq X_{1} \subseteq X_{2} \cdots$. Then we define the $\alpha$-phantom maps of $X$ inductively as follows.
We let $$\mathrm{Ph}^{0}(X, Y) = \bigcap_{i \in \mathbb{N}} \Ran([(X, X_{i}), (Y, *)] \rightarrow [(X, *), (Y, *)])$$
For each successor ordinal $\alpha$, we let 
$$\mathrm{Ph}^{\alpha+1}(X, Y) = \bigcap_{i \in \mathbb{N}} \Ran(\mathrm{Ph}^{\alpha}[(X, X_{i}), (Y, *)] \rightarrow [(X, *), (Y, *)])$$
and for a successor ordinal $\lambda$, we let 
$$
\mathrm{Ph}^{\lambda} = \bigcap_{\lambda' < \lambda} \mathrm{Ph}^{\lambda'}(X, Y).
$$

\end{definition}

The following is immediate from \cref{cor6.12} and \cref{thm5.1}:

\begin{corollary}\label{cor6.14}
Suppose that $X$ is a locally compact Polish space and $K$ is a locally finite countable simplicial complex. Suppose that the following conditions are satisfied:
\begin{enumerate}
\item{$\check{H}^{m+1}(X, X_{l}; \pi_{m}|K|) = \check{H}^{m}(X, X_{l}; \pi_{m}|K|) $ for each $m \ge n+1$ and $l \ge 1$}
\item{$\check{H}^{m+1}(X; \pi_{m}|K|) = \check{H}^{m+1}(X; \pi_{m+1}|K|) = 0$ for $m \ge n$}
\item{$K$ is $n-1$-connected.}
\item{$X$ has finite covering dimension}
\end{enumerate}
Then there is a natural bijection of sets
$$
\mathrm{Ph}^{\alpha}(X, |K|) \cong \mathrm{Ph}^{\alpha}(\check{H}^{n}(X, \pi_{n}|K|))
$$
for each $\alpha < \omega_1$.
\end{corollary}

\section{Examples of Borel Complexity for the Homotopy Relation}\label{sec7}

In this section, we will give examples of the potential Borel complexities for homotopy classes of map $[X, |K|]$, where $K$ is a simplicial complex. 
To do this, we will use the calculation of the potential complexities of $\mathrm{Ext}$ described in \cite{casarosa_bootstrap_2025}. 
\\

We start by stating a special case of  \cite[Theorem 13.25]{casarosa_bootstrap_2025} which will be used to give examples of the potential complexities of the homotopy relation:
\begin{theorem}\label{thm7.1}
Suppose that $\Gamma =  \mathbf{\Sigma}_{1 + \lambda + 1}^{0}, D(\Pi_{1 + \lambda + m}^{0}), \Pi_{1}^{0}, \Pi_{\lambda}^{0}, \Pi_{1 + \lambda + m + 1}^{0}$ for some limit ordinal $\lambda < \omega_1$ and $1 \le m < \infty$.
Then there exists a countable abelian group $A_{\Gamma}$ such that 
$$
\mathrm{Ext}^{1}(A_{\Gamma}, \mathbb{Z})
$$
has potential complexity $\Gamma$. 
\end{theorem}

For $X$ a CW complex, we will denote by $H_{n}(X)$ the usual singular homology of $X$. 
If $X$ is a compact Polish space, we will denote by $H_{n}^{st}$, the Steenrod homology of $X$. Note that if $X$ is a CW complex $H_{n}^{st}(X) \cong H_{n}(X)$.

\begin{lemma}\label{lem7.2}

Suppose that $X$ is a locally compact Polish space with an exhaustion $X_{0} \subseteq X_{1} \subseteq X_{2} \cdots$, 
such that $H_{i}^{st}(X_{j}) = 0$ for all $i \neq n-1, j \in \mathbb{N}$.
Then there is a definable bijection $\check{H}^{n}(X; G) \simeq \mathrm{lim}^{1}\mathrm{Hom}(H_{n-1}^{st}(X_{m}); G)$.
 \end{lemma}
\begin{proof}
First, we note that by the definable version of the Huber isomorphism (see \cref{thm2.7} above), \cite[Theorem 7.10]{bergfalk_definable_2024} and \cite[Theorem 7.14]{bergfalk_definable_2024}:
$$
\check{H}^{n}(X; G) \cong [X, K(n, G)] \cong \mathrm{lim}^{1}([\Sigma X_{m}, K(n, G)]) \cong \mathrm{lim}^{1}(\check{H}^{n-1}(X_{m}; G) )
$$
On the other hand, the definable version of the universal coefficient theorem from \cite{lupini_definable_2020} implies that we have a definable bijection:
$$
 \mathrm{lim}^{1}(\check{H}^{n-1}(X_{m}; G) ) \cong  \mathrm{lim}^{1} \Hom(H_{n-1}^{st}(X_{m}), G)
$$

\end{proof}

\begin{definition}\label{def7.3} A Moore space $M(A, n)$ is a CW complex that's homology satisfies the following: 
$$
H_{m}(M(A, n)) =
\left\{
	\begin{array}{ll}
		A  & \mbox{if } m = n \\
		\mathbb{Z} & \mbox{if } m = 0 \\
		0 & \mbox{if } otherwise
	\end{array}
\right.
$$
\end{definition}
\begin{construction}\label{con7.4}
We will explain how to construct a Moore space $M(A, n)$, where $A$ is an abelian group. Then $A$ has a presentation given by a short exact sequence 
\begin{equation}\label{eq1}
0 \rightarrow F_{1} \rightarrow F_{0} \rightarrow A \rightarrow 0
\end{equation}
where $F_{1}, F_{0}$ are free abelian groups. Write $J_{1}$ and $J_{0}$ for sets of minimal generators for $F_{1}$ and $F_{0}$. 

We will define a space $X$ by taking the n-skeleton to be:
$$
X^{n} = \bigvee_{\alpha \in J_{0}} S^{n}_{\alpha} 
$$
We are now going to define an attaching map for the $n+1$-cells
$$
\coprod_{\beta \in J_{1}} S^{n}_{\beta} \rightarrow X^{n} 
$$
by specifying it on each summand. 
We can write each generator of $\beta$ of $J_{1}$ in $J_{0}$ as a linear combination of generators of 
$$
\sum_{i=1}^{n} n_{i} \alpha_{i}  = \beta
$$

The attaching map on the summand $S^{n}_{\beta}$ is given by:
$$
S^{n}_{\beta} \rightarrow (\bigvee_{i=1}^{s} S^{n}_{\alpha_{i}}) \xrightarrow{\bigvee \phi_{\alpha_i}} (\bigvee_{i=1}^{n} S^{n}_{\alpha_{i}}) 
$$
where the first map is the pinch map obtained by pinching $s-1$ circles on $S^{n}$ to points, and each $\phi_{\alpha_{i}}$ is a map of degree $n_{i}$.
The cellular cochain complex of this space is:
$$
\cdots 0 \rightarrow F_{1} \rightarrow F_{0} \rightarrow 0 \cdots,
$$
concentrated in degrees $n+1, n$
from which we conclude that $X = M(A, n)$.
\end{construction}

\begin{lemma}\label{lem7.5}
Suppose that $A$ is a countable torsion-free abelian group. Then $M(A, n) = \mathrm{colim}_{i \in \mathbb{N}} M(A_{i}, n)$ where each $A_{i}$ is a finitely generated torsion-free abelian group. 

\end{lemma}
\begin{proof}
Let us write $A$ as a short exact sequence $$0 \rightarrow \bigoplus_{i = 0}^{Q} \mathbb{Z} \xrightarrow{i} \bigoplus_{i = 0}^{\infty} \mathbb{Z} \rightarrow A \rightarrow 0 ,$$ where $Q = m$ for some integer $m$ or $Q = \infty$. For each $m <= Q$, let $k_{m}$ be the largest integer such that $i (\bigoplus_{i=0}^{k_{m}} \mathbb{Z} ) \subseteq \bigoplus_{i = 0}^{m} \mathbb{Z} \subseteq \bigoplus_{i = 0}^{\infty} \mathbb{Z} $. Consider the group $A_{m}$ with presentation
$$
0 \rightarrow \bigoplus_{i=0}^{k_{m}} \mathbb{Z} \xrightarrow{i_{m}} \bigoplus_{i=0}^{m} \mathbb{Z} \rightarrow A_{m} \rightarrow 0
$$
where $i_{m}$ is the restriction of $i$. This group is clearly torsion-free, since $A$ is. By the construction of Moore spaces, we have obvious inclusions
$$
M(A_{0}, n) \subseteq M(A_{1}, n) \subseteq \cdots
$$
and since homology commutes with filtered colimits, we have 
$$
\mathrm{colim}_{m \in \mathbb{N}} M(A_{m}, n) = M(A, n).
$$
\end{proof}

\begin{theorem}\label{thm7.6}
Suppose that $\Gamma =  \mathbf{\Sigma}_{1 + \lambda + 1}^{0}, D(\Pi_{1 + \lambda + m}^{0}), \Pi_{1}^{0}, \Pi_{\lambda}^{0}, \Pi_{1 + \lambda + m + 1}^{0}$ for some limit ordinal $\lambda < \omega_1$ and $1 \le m < \infty$.
Then there exists a countable abelian group $A_{\Gamma}$ such that for each $n \ge 3$, the homotopy relation on $C(M(A_{\Gamma}, n-1), S^{n})$ has potential complexity $\Gamma$.

Moreover, if $n \ge 4$ and $X$ is a locally compact Polish space, we have a bijection:
$$
\mathrm{Ph}^{\alpha}[X, S^{n}] \cong \mathrm{Ph}^{\alpha} \check{H}^{n}(X; \mathbb{Z})
$$
\end{theorem}
\begin{proof}
By the construction of Moore spaces in \cref{con7.4} and \cref{lem7.5}, we see that $M(A, n-1)$ has an exhaustion consisting of spaces $$ M(A_{1}, n-1) \subseteq \cdots \subseteq M(A_{m}, n-1) \subseteq  M(A_{m+1}, n-1) \cdots$$ that satisfy the hypotheses of \cref{lem7.2}.
The universal coefficient theorem implies that $M(A, n-1)$ satisfies the third hypotheses of \cref{thm5.1}. By construction, the Moore space $M(A, n-1)$ is an n-dimensional CW-complex, which is locally finite and countable. Thus, it can be embedded in $\mathbb{R}^{2n+1}$ by \cite[Theorem 1.5.15]{fritsch_cellular_1990}, and thus has finite covering dimension by \cite[Example 1, pg 59]{fritsch_cellular_1990}.

By \cref{thm5.1}, \cref{lem7.2} and \cite[Theorem 7.4]{bergfalk_definable_2020} we have a definable bijection:
$$
[M(A, n-1), S^{n}] \cong \check{H}^{n}(X; \mathbb{Z}) \cong lim^{1} \mathrm{Hom}(A_{n}, \mathbb{Z}) \cong \mathrm{Ext}^{1}(A; \mathbb{Z})
$$
\cref{thm7.1} implies that we can find a group $A_{\Gamma}$ such that $\mathrm{Ext}^{1}(A_{\Gamma}; \mathbb{Z})$ has potential complexity $\Gamma$.

For the latter statement, we need to apply \cref{cor6.14}. By the universal coefficient theorem, and the long exact sequence of a pair in cohomology, the only condition that is not trivial to verify is the condition that 
$$
\check{H}^{n+1}(M(A, n-1), M(A_{m}, n-1); \pi_{n+1}(S^{n})) =  \check{H}^{n+1}(M(A, n-1), M(A_{m}, n-1); \mathbb{Z} / 2 \mathbb{Z}) = 0
$$
(Note that $\pi_{n+1}(S^{n}) \cong \mathbb{Z} / 2 \mathbb{Z}$ for $n \ge 4$ by \cite[pg. 485]{fomenko_homotopical_2016}).
To verify this condition, the long exact sequence of cohomology reduces to
$$
0 \leftarrow \check{H}^{n+1}(M(A, n-1), M(A_{m}, n-1); \pi_{n+1}(S^{n})) \leftarrow \check{H}^{n}(M(A_{m}, n-1); \mathbb{Z} / 2 \mathbb{Z} ) \leftarrow  \check{H}^{n}(M(A, n-1); \mathbb{Z} / 2 \mathbb{Z} ) \cdots
$$
By the universal coefficient theorem, we have that 
$$
\check{H}^{n}(M(A_{m}, n-1); \mathbb{Z} / 2 \mathbb{Z} ) \cong \mathrm{Ext}^{1}(A_{m}, \mathbb{Z} / 2 \mathbb{Z}) \cong 0
$$
where the latter isomorphism follows since $A_{m}$ can be chosen to be torsion-free by \cref{lem7.5}.
\end{proof}

\begin{definition}\label{def7.7} (see \cite[Definition 13.15]{casarosa_bootstrap_2025})
Suppose that $G/H$ is a group with a Polish cover and $\alpha$ is a continuous action of $\Lambda$ on $G$ that leaves $H$ setwise invariant. We define $\mathrm{Out}(\Lambda \curvearrowright G/H)$ to be the quotient $G/E$, where $E$ is the idealistic equivalence relation given by 
$$
\{ (g , g') \in G \times G : \exists \lambda \in \Lambda, \lambda H g = H g'  \}
$$ 
\end{definition}

\begin{example}\label{exam7.8}
Suppose that $X$ is a locally compact Polish space and $P$ is a polyhedron. We say that a map $f : X \rightarrow X$ is a \emph{homotopy automorphism} if there exists some map $g$ such that $f \circ g \simeq id \simeq g \circ f$. 
Let us write $\mathrm{haut}(X)$ for the subset of $[X, X]$ consisting of homotopy automorphisms. This is a group under composition. Precomposition determines an action
$$
\mathrm{haut}(X) \curvearrowright [X, P]
$$
and we have an idealistic equivalence relation $R$ on $C(X, P)$ given by $ f \simeq g$ if
$$
\gamma \cdot f \simeq g
$$
for some $\gamma \in \mathrm{haut}(X)$ (c.f. \cite[Section 8.7]{bergfalk_definable_2024}). We can thus ask about the potential complexity of this equivalence relation. 

Consider the case $X = M(A, n-1)$ and $P = S^{n}$ for $n \ge 2$.
By the proof of \cref{thm7.6}, we have a definable isomorphism:
\begin{equation}\label{ningbo}
[X, P] \cong \mathrm{Ext}^{1}(H_{n-1}(X), \mathbb{Z}) \cong  \mathrm{Ext}^{1}(A, \mathbb{Z})
\end{equation}

On the other hand, by the discussion of \cite[Definition 13.15]{casarosa_bootstrap_2025}, there is an idealistic equivalence relation on 
under which the action 
$$
\mathrm{haut}(X) \curvearrowright [X, P]
$$
can be identified with the action 
$$\mathrm{haut}(X) \curvearrowright \mathrm{Ext}^{1}(H_{n}(X), \mathbb{Z})$$
given by 
$$
f \cdot a = H_{n}(f)^{*}(a)
$$
In the case that $X$ is a Moore space \cite[Corollary 1.3.10]{baues_homology_1996} states that there is a short exact sequence 
$$
0 \rightarrow \mathrm{Ext}^{1}(A, \Gamma^{n}_{1}(A)) \rightarrow \mathrm{haut}(X) \xrightarrow{H_{n}(-)} \mathrm{Aut}(A) \rightarrow 0
$$
where $\Gamma^{1}_{n}(A)$ is the functor from \cite[Definition 1.3.7]{baues_homology_1996}.

Combining this with \cref{ningbo}, we see that there is a definable bijection:
$$
C(X, S^{n})/R \cong \mathrm{Out}(\mathrm{Aut}(A) \curvearrowright \mathrm{Ext}^{1}(A, \mathbb{Z}))
$$

It then follows from \cite[Theorem 13.25]{casarosa_bootstrap_2025} that for each limit $\lambda < \omega_{1}$ and $ m < \omega$, $m \neq 1$ there exists a space $X$ such that the relation $R$ on $C(X, S^{n})$ is $\Sigma_{\lambda +m + 1}^{0}$-definable but not $\Sigma_{\lambda +m}^{0}$-definable. 
\end{example}

\begin{remark}\label{rmk7.9}

We could have used the definable Huber isomorphism from \cref{thm2.6} and Eilenberg-Maclane spaces to give an example when the homotopy relation has arbitrary complexity, but \cref{thm5.1} allows on to produce a wider range of examples; in our case, the example where the target space is an $n$-sphere. 
\end{remark}

\begin{construction}\label{con7.10}
Suppose that we have a tower of spaces 
$$
X_{0} \xrightarrow{f_{0}} X_{1} \xrightarrow{f_{1}} X_{2} \xrightarrow{f_{2}} \cdots 
$$
Then the \emph{mapping telescope} associated with the tower is the quotient space 
$$
\mathrm{hocolim} \, X_{i}  = \coprod_{n \in \mathbb{N}} [n, n+1] \times X_{n} / \sim
$$
where $\sim$ is the equivalence relation that identifies $(f_{n}(x), n+1)$ with $(x, n)$. Alternatively, it is the quotient space
$$
\mathrm{hocolim} \, X_{n} := \coprod_{n \in \mathbb{N}} M_{f_{n}} / \sim
$$
where $M_{f_{n}}$ denotes the mapping cylinder of $f_{n}$ and $\sim$ is the equivalence relation which identifies the copy of $X_{n}$ in $M_{f_{n-1}}$ with $X_{n} \times 0$ in $M_{f_{n}}$ for each $n > 0$.

The mapping cylinder has a filtration $\tilde{X_{n}}$ given by 
$$
\tilde{X_{n}} = \{ (x, t) \in \mathrm{hocolim} \, X_{n} : t \in [0, n] \}
$$
By using the deformation retract of a mapping cylinder onto the target space, it is easy to see that $\tilde{X_{n}}$ is homotopy equivalent to $X_{n}$.

\end{construction}

\begin{example}\label{exam7.11}

Let $p$ be a prime number. 
Consider the mapping telescope of the collection 
$$
S^{d}_{1} \xrightarrow{m_{p}} S^{d}_{2} \xrightarrow{m_{p}} S^{d}_{3} \cdots
$$
where $m_{p}$ is a map of degree $p$ and $S^{d}_{i}$ a sphere of degree $p$. Consider the mapping telescope 
$$
\mathrm{hocolim} \, S^{d}_{i}
$$
with exhaustion $\{ \tilde{S^{d}_{i}} \}$ as in \cref{con7.10}. 
This mapping cylinder is homeomorphic to the Solenoid complement (see \cite[Section 8.6]{bergfalk_definable_2024}).

Since $\tilde{S^{d}_{i}}$ is a $d$-sphere, we can apply the proof of \cref{thm7.6} to show that we have a definable bijection:
$$
\Ext^{1}(\mathbb{Z}[1/p], \pi_{d+1}(|K|))= \Ext^{1}(H_{d-1}(\mathrm{hocolim} \,  S^{d}_{i}), \pi_{d+1}(|K|)) \cong [\mathrm{hocolim} \, (S^{d}_{i}), |K|]
$$
for any $d+1$-connected, locally finite countable complex $K$.

It follows from \cite[Theorem 9.38 and Theorem 13.12]{casarosa_bootstrap_2025}  $\Ext^{1}(\mathbb{Z}[1/p], \pi_{d+1}(|K|))$ has potential complexity $\Sigma_{2}^{0}$. Thus, the homotopy relation on $C(\mathrm{hocolim} \, S^{d}_{i}, |K|)$ has potential complexity $\Sigma_{2}^{0}$ for each locally finite, countable simplicial complex $K$. 
\end{example}

\bibliographystyle{alpha}
\bibliography{bibliographymartino}

\end{document}